\documentclass[pdflatex,sn-mathphys-num]{sn-jnl}

\usepackage{graphicx}%
\usepackage{multirow}%
\usepackage{amsmath,amssymb,amsfonts}%
\usepackage{amsthm}%
\usepackage{mathrsfs}%
\usepackage[title]{appendix}%
\usepackage{xcolor}%
\usepackage{textcomp}%
\usepackage{manyfoot}%
\usepackage{booktabs}%
\usepackage{algorithm}%
\usepackage{algorithmicx}%
\usepackage{algpseudocode}%
\usepackage{listings}%

\usepackage{longtable}
\usepackage{booktabs}
\usepackage{array}

\theoremstyle{thmstyleone}%
\newtheorem{theorem}{Theorem}
\newtheorem{proposition}[theorem]{Proposition}%
\newtheorem{lemma}[theorem]{Lemma}
\newtheorem{corollary}[theorem]{Corollary}

\theoremstyle{thmstyletwo}%
\newtheorem{example}{Example}%
\newtheorem{remark}{Remark}%

\theoremstyle{thmstylethree}%
\newtheorem{definition}{Definition}%

\newcommand{\F}{\mathbb{F}}
\newcommand{\wt}{\mbox{\rm wt}}
\newcommand{\rank}{\mbox{\rm rank}}

\begin{document}

\title{Enumeration and Monotonicity of Linear Codes with Prescribed Hull Dimensions}


\author*[1]{\fnm{Stefka} \sur{Bouyuklieva}}\email{stefka@ts.uni-vt.bg}

\author[2]{\fnm{Iliya} \sur{Bouyukliev}}\email{iliyab@math.bas.bg}

\author[3]{\fnm{Keita} \sur{Ishizuka}}\email{keitaishizuka1994@gmail.com}

\author[4]{\fnm{Ferruh} \sur{\"{O}zbudak}}\email{ferruh.ozbudak@sabanciuniv.edu}

\affil[1]{\orgdiv{Faculty of Mathematics and Informatics}, \orgname{St. Cyril and St. Methodius University of Veliko Tarnovo}, \city{Veliko Tarnovo}, \postcode{5000},  \country{Bulgaria}}

\affil[2]{\orgdiv{Institute of Mathematics and Informatics}, \orgname{Bulgarian Academy of Sciences}, \city{Veliko Tarnovo}, \postcode{5000}, \country{Bulgaria}}

\affil[3]{\orgdiv{Information Theory R\&D Center}, \orgname{Mitsubishi Electric Corporation}, \city{Kanagawa}, \postcode{2510052}, \country{Japan}}

\affil[4]{\orgdiv{Faculty of Engineering and Natural Sciences}, \orgname{Sabanc{\i} University}, \city{Istanbul}, \postcode{34956}, \country{T\"{u}rkiye}}


\abstract{The hull of a linear code is the intersection of the code with its dual and
depends on the choice of inner product. We study how the numbers of linear codes
with prescribed Euclidean, Hermitian, and symplectic hull dimensions vary with
the hull dimension. Building on closed enumeration formulae, we derive exact
ratio decompositions for the Hermitian and symplectic cases, determine the
parameter ranges in which monotonicity holds, and identify the exceptional
families. We also compare the three inner products and complement the theoretical
analysis with computational classification results for quaternary codes.}





\maketitle

\section{Introduction}
\label{intro}

Counting subspaces of a vector space with prescribed structural invariants
is a classical theme in combinatorics and algebra. Let $C$ be a linear $[n,k]_q$ code over the finite field $\mathbb{F}_q$, that is, a
$k$-dimensional linear subspace of the vector space $\mathbb{F}_q^n$.
Assume that $\mathbb{F}_q^n$ is equipped with an inner product $\langle \cdot,\cdot\rangle$ which may be a bilinear, Hermitian, or alternating form.
The \emph{dual code} (or \emph{orthogonal code}) of $C$ is defined by
\[
C^{\perp}=\{u \in \mathbb{F}_q^n : \langle u,v \rangle =0 \text{ for all } v \in C\}.
\]
The intersection
\[
\mathcal{H}(C)=C\cap C^{\perp}
\]
is called the \emph{hull} of $C$ with respect to the given inner product. The hull dimension can vary
between two extremes. If $\mathcal{H}(C)=\{0\}$, then $C$ is called an \emph{LCD (Linear Complementary Dual) code}.
If $C\subseteq C^{\perp}$, then $C$ is said to be \emph{self-orthogonal}.

The number of linear $[n,k]_q$ codes with a prescribed hull dimension is given by a closed enumeration formula (often referred to as a mass formula in some papers). Such formulas were recently obtained by several authors through the action of the corresponding classical group on the Grassmannian (Carlet et al. \cite{CarletLCD2019}, Liu and Wang \cite{Liu_Wang}, Li, Shi, and Ling \cite{LiShiLing}, \cite{LiShiLing_further}).

A natural follow-up question is how this count varies with the hull dimension.
Bouyuklieva, Bouyukliev, and \"{O}zbudak (BBO) \cite{BBO2025} initiated this
systematic study for the Euclidean inner product, establishing a ratio decomposition
$A_{n,k,\ell,q} = \alpha_{n,k,\ell,q} (q^{\ell+1}-1)A_{n,k,\ell+1,q}$ with $\alpha\ge 1/2$ whenever $A_{n,k,\ell+1,q}>0$,
the bound being attained at certain parity-restricted parameters governed
by the quadratic character of $\mathbb{F}_q$. Their proof builds on the closed-form enumeration
formulae of Li, Shi, and Ling \cite{LiShiLing} (with the earlier formula of Sendrier \cite{Sendrier_hull}
being too unwieldy for this purpose). The analysis first splits according to
the parity of $q$ (odd vs. even), then within each into four sub-cases by the
parities of $n$ and $k-\ell$; the equality case $\alpha= 1/2$ additionally requires the
residue of $q$ modulo 4 to be specified when $n$ is even.

The natural next step is to ask whether the BBO framework extends
to other inner products. The Hermitian and symplectic inner products
are the remaining standard choices, and their hull dimensions arise prominently
in quantum coding theory: Hermitian-hull-graded codes yield
entanglement-assisted quantum error-correcting codes \cite{Guenda}, while symplectic
self-orthogonal codes underlie the stabilizer construction \cite{Ketkar,WildeBrun}.
Beyond their intrinsic interest, comparing the three inner products
side by side reveals a striking dichotomy that we trace to the Witt
classification of the underlying form.

\textbf{Contribution.} In this paper, we extend the BBO framework to the Hermitian
and symplectic inner products and undertake a comparative analysis
of all three settings. Our main contributions are:
\begin{itemize}
\item[(i)] (\emph{Hermitian hull} ) We derive a closed-form expression for the Hermitian
ratio factor $\alpha^H$ and prove algebraically that $\alpha^H\ge q/(q + 1) \ge 2/3$
in all valid cases. The strict inequality $\alpha^H > 1$ holds outside an
explicit boundary family ($\ell=0$, $n$ even, $k\in\{1,n-1\}$). Both the
lower bound and the exceptional set are strictly stronger than their
Euclidean counterparts (Theorem~\ref{main_hermitian}).
\item[(ii)] (\emph{Symplectic hull} ) We derive the analogous closed form expression for the symplectic ratio factor
$\alpha^S_{2n,k,\ell,q}$, which satisfies $\alpha^S_{2n,k,\ell,q}>\frac{1}{q^2}$.
Here the hull dimension changes in steps of two, and
$A^S_{2n,k,\ell,q}=\alpha^S_{2n,k,\ell,q}(q^{\ell+1}-1)(q^{\ell+2}-1)A^S_{2n,k,\ell+2,q}$.
Hence the inequality $A^S_{2n,k,\ell,q}>(q^{\ell+1}-1)A^S_{2n,k,\ell+2,q}$, which is the analogue of the Euclidean inequality in Theorem~\ref{thm:ineq}, holds if and only if $\alpha^S_{2n,k,\ell,q}(q^{\ell+2}-1)>1$. We prove that this is the case for all valid parameters except when $\ell=0$ and $4\le k\le 2n-4$, where $\alpha^S_{2n,k,0,q}(q^{2}-1)<1$ for every $q$. The weaker inequality $A^S_{2n,k,\ell,q}>A^S_{2n,k,\ell+2,q}$, which says that the number of codes decreases when the hull dimension increases, holds for all valid parameters except when $q=2$, $\ell=0$ and $4\le k\le 2n-4$ ($k$ even); for these parameters $A^S_{2n,k,0,2}<A^S_{2n,k,2,2}$ (Theorems~\ref{main_symplectic} and~\ref{thm:SHull}).

\item[(iii)] (\emph{Comparison and applications}) We juxtapose the three settings in a comparative table, relating the qualitative differences in $\alpha^*$ 
to the orbit structure of the corresponding classical group actions and to the Witt classification of the underlying forms. We further compare the numbers of codes 
with the same parameters under the three inner products and discuss applications to entanglement-assisted quantum codes via the Guenda--Jitman--Gulliver \cite{Guenda} and Wilde--Brun \cite{WildeBrun} constructions.
\end{itemize}

\textbf{Organization.} Section~\ref{sect:preliminaries} fixes the notation and
recalls the required enumeration formulae. Sections~\ref{sect:Hermitian}
and~\ref{sect:symplectic} establish the Hermitian and symplectic ratio
decompositions, respectively. The subsequent sections compare the three inner
products, present the computational classification results, and conclude the paper.

\section{Preliminaries}
\label{sect:preliminaries}

In this section, we present the basic definitions and enumeration formulae used
throughout the paper.

A linear $[n,k,d]_q$ code $\mathcal{C}$ is a $k$-dimensional subspace of the
vector space $\F_q^n$, and $d$ is the
smallest weight among all non-zero codewords of $\mathcal{C}$, called the \textit{minimum weight} (or minimum distance) of the code. 

In this work, we consider the following three inner products.

\begin{itemize}
\item \emph{Euclidean (standard) inner product}:
\[
x\cdot y=\sum_{i=1}^n x_i y_i,
\]
for any $x=(x_1,\ldots,x_n)$ and $y=(y_1,\ldots,y_n)\in\mathbb{F}_q^n$. It is a symmetric bilinear form in $\F_q^n$.

\item \emph{Hermitian inner product}:
\[
\langle x,y\rangle_H=\sum_{i=1}^n x_i y_i^{\,q},
\]
for any $x=(x_1,\ldots,x_n)$ and $y=(y_1,\ldots,y_n)\in\mathbb{F}_{q^2}^n$.
This product is defined only over fields with $q^2$ elements,
where $q$ is a prime power, and represents a Hermitian form in the considered vector space.

\item \emph{Symplectic inner product}:
\[
\langle x,y\rangle_S=\sum_{i=1}^n (x_i y_{n+i}-x_{n+i} y_i)=x\Omega_n y^T,
\]
where $x=(x_1,\ldots,x_{2n})$, $y=(y_1,\ldots,y_{2n})\in\mathbb{F}_{q}^{2n}$, and $\Omega_n=\begin{pmatrix} O& I_n\\
-I_n&O\\
\end{pmatrix}$.
The symplectic inner product is defined only in vector spaces of even
dimension. It is an alternating form in $\F_q^{2n}$.
\end{itemize}

For a linear code $C$ over $\F_q$ (resp. $\F_{q^2}$), the Euclidean dual, Hermitian dual,
and symplectic dual are denoted by $C^{\perp_E}$, $C^{\perp_H}$, and $C^{\perp_S}$, respectively. For each of the three inner products, the hull is defined as the intersection
of the code with its dual:
\begin{eqnarray*}
  \mathcal{H}_E(C) &=& C\cap C^{\perp_E} \\
  \mathcal{H}_H(C) &=& C\cap C^{\perp_H} \\
  \mathcal{H}_S(C) &=& C\cap C^{\perp_S}. 
\end{eqnarray*}

The hull dimension interpolates between two extremes: $C$ is called a linear
complementary dual (LCD) code if its hull is trivial (introduced by
Massey \cite{Massey}), and self-orthogonal if $C$ is contained in the appropriate dual.
Binary LCD codes have applications to side-channel attack countermeasures \cite{Carlet_Guilley} and decoding complexity, while self-orthogonal codes underlie the
construction of quantum error-correcting codes via the stabilizer and
entanglement-assisted frameworks \cite{Guenda,Ketkar}.
The hull dimension can be read off from any generator matrix:

\begin{proposition}
Let $C$ be a linear code with generator matrix $G$.

\begin{enumerate}
\item[\textup{(1)}] {\rm\cite[Proposition 3.2]{Guenda}} If $C$ is an $[n,k]_{q^2}$ code, then
$\dim(\mathcal{H}_H(C)) = k-\rank(G\overline{G}^{\,T})$, where $\overline{G}$ is obtained from $G$ by applying the Frobenius automorphism $x\mapsto x^q$ entrywise. In particular, $C$ is Hermitian LCD if and only if
$G\overline{G}^{\,T}$ is nonsingular, and $C$ is Hermitian self-orthogonal if and only if
$G\overline{G}^{\,T}=O$.

\item[\textup{(2)}] {\rm\cite[Proposition 2.1(2)]{LiShiLing_further}} If $C$ is an $[2n,k]_q$ code, then $\dim(\mathcal{H}_S(C)) = k-\rank(G\Omega_nG^T)$.
In particular, $C$ is symplectic LCD if and only if $G\Omega_nG^T$ is nonsingular,
and $C$ is symplectic self-orthogonal if and only if $G\Omega_nG^T=O$.
Since $G\Omega_nG^T$ is alternating, $\rank(G\Omega_nG^T)$ is even. Hence $k-\dim(\mathcal{H}_S(C))$ is even.
\end{enumerate}
\end{proposition}

Moreover, $0\le \dim(\mathcal{H}_H(C))\le \min\{k,n-k\}$, and $0\le \dim(\mathcal{H}_S(C))\le \min\{k,2n-k\}$, with the additional constraint that $k-\dim(\mathcal{H}_S(C))$
is even.

To illustrate the structure of these hulls for a single code,
we present the following example.

\begin{example}
Let $C$ be a linear $[6,2]$ code over $\mathbb{F}_4=\{0,1,\omega,\omega^2\}$,
where $1+\omega+\omega^2=0$, generated by
\[
G=
\begin{pmatrix}
1 & 0 & 1 & 1 & 1 & 0\\
0 & 1 & 1 & \omega & \omega^2 & 1
\end{pmatrix}.
\]
The code $C$ is Euclidean self-orthogonal, and hence
$\mathcal{H}_E(C)=C$.
The two rows of $G$ are Hermitian orthogonal to each other, but the second
row is not Hermitian orthogonal to itself. Consequently, $\mathcal{H}_H(C)=\langle (1\,0\,1\,1\,1\,0)\rangle$.
On the other hand, the two rows are not symplectic orthogonal, and therefore
$\mathcal{H}_S(C)=\{0\}$.

Hence $C$ is Euclidean self-orthogonal and symplectic LCD.
\end{example}

This work continues the authors' previous research on the enumeration of
linear codes of fixed length and dimension over a finite field with
prescribed Euclidean hull dimension. In \cite{BBO2025}, we studied the
sequence
\[
\mathcal{S}_{n,k,q}=
\bigl(A_{n,k,0,q}^E,A_{n,k,1,q}^E,\ldots,A_{n,k,\ell,q}^E,\ldots,A_{n,k,k,q}^E\bigr),
\]
where $A_{n,k,\ell,q}^E$ denotes the number of linear $[n,k]_q$ codes
(equivalently, $[n,n-k]_q$ codes) whose Euclidean hull has dimension $\ell$.
Our main result in \cite{BBO2025} is the following theorem.

\begin{theorem}\cite[Theorem 12]{BBO2025}\label{thm:ineq}
Let $\ell+1\le k\le n-\ell-1$. Then
$A_{n,k,\ell,q}^E>(q^{\ell+1}-1)A_{n,k,\ell+1,q}^E$
for all possible values of $q$, $n$, $k$ and $\ell$ except when $q$ is odd, $n$ is even, $(-1)^{n/2}$ is a square in $\mathbb{F}_q^*$, and $k-\ell$ is odd, in which case
$A_{n,k,\ell,q}^E\ge\frac{q^{\ell+1}-1}{2}A_{n,k,\ell+1,q}^E$.
\end{theorem}

A more detailed version is given in \cite[Lemmas 8 and 10]{BBO2025}.

\begin{theorem}\label{thm:Ealpha}
Let $q$ be a prime power, $n$ and $k$ be
positive integers with $k\le n/2$, and $\ell$ be an integer such that $0\le\ell\le k-1$.
Assume $A_{n,k,\ell+1,q}^E>0$. Then
\[
A_{n,k,\ell,q}^E = \alpha_{n,k,\ell,q}^E (q^{\ell+1}-1)A_{n,k,\ell+1,q}^E,
\]
where
\begin{equation}
\alpha_{n,k,\ell,q}^E=\left\{\begin{array}{ll}
\frac{q^{n-k-\ell}}{q^{n-k-\ell}-1}&\mbox{if} \ n \ \mbox{is odd}, k-\ell \ \mbox{is odd},\\
\frac{q^{k-\ell}}{q^{k-\ell}-1}&\mbox{if} \ n \ \mbox{is odd}, k-\ell \ \mbox{is even},\\
\frac{q^{n-\ell-1}}{q^{n-\ell-1}-1}&\mbox{if} \ n \ \mbox{is even}, k-\ell \ \mbox{is odd}, \ \mbox{and} \ q \ \mbox{is even},\\
\frac{q^{n/2-\ell-1}}{q^{n/2-\ell-1}+\eta((-1)^{n/2})}&\mbox{if} \ n \ \mbox{is even}, k-\ell \ \mbox{is odd}, \ \mbox{and} \ q \ \mbox{is odd},\\
\frac{q^{n-\ell}-1}{q^\ell(q^{n-k-\ell}-1)(q^{k-\ell}-1)}&\mbox{if} \ n \ \mbox{is even}, k-\ell \ \mbox{is even}, \ \mbox{and} \ q \ \mbox{is even},\\
\frac{q^{n/2-\ell}(q^{n/2-\ell}+\eta((-1)^{n/2}))}{(q^{n-k-\ell}-1)(q^{k-\ell}-1)}&\mbox{if} \ n \ \mbox{is even}, k-\ell \ \mbox{is even}, \ \mbox{and} \ q \ \mbox{is odd}.
\end{array}\right.
\end{equation}
If $q$ is odd, $n$ is even, $(-1)^{n/2}$ is a square in $\mathbb{F}_q^*$, and $k-\ell$ is odd, then $1/2\le\alpha_{n,k,\ell,q}^E<1$ with equality in the lower bound if and only if $k=n/2$ and $\ell=k-1$. In all other cases
$\alpha_{n,k,\ell,q}^E > 1$.
The excluded case occurs precisely when $q$ is odd, $n=2k$, $\ell=k-1$, and $\eta((-1)^k)=-1$.
In this case $A_{n,k,\ell+1,q}^E=0$ and $A_{n,k,\ell,q}^E>0$, so Theorem~\ref{thm:ineq} still applies without forming a ratio.
\end{theorem}

The formulae for odd characteristic depend on the quadratic character $\eta$ of $\F_q$. Recall that $\eta(c)=1$ if $c$ is a square in $\F_q^*$, and $\eta(c)=-1$ otherwise ($c\neq 0$).

\begin{remark} The statement of Theorem 12 in \cite{BBO2025} (and Condition * in the introduction of \cite{BBO2025}) contains a misprint in the description of the exceptional case, which appears there as ``$n\equiv 2\pmod 4$ and $-1$ is not a square in $\mathbb{F}_q^*$'', whereas the ratio formula in Lemma 8 gives the correct exceptional condition. The correct condition is the one in \cite[Lemma 8]{BBO2025}, which we have used in Theorems \ref{thm:ineq} and \ref{thm:Ealpha}: $n\equiv 0\pmod 4$, or $n\equiv 2\pmod 4$ and $q\equiv 1\pmod 4$, that is, $(-1)^{n/2}$ is a square in $\mathbb{F}_q^*$.
\end{remark}

We expect that an analogous dependence also holds for Hermitian and
symplectic hulls. In particular, we show that the enumeration of linear codes
with respect to the dimensions of their Hermitian and symplectic hulls
exhibits behavior similar to that observed in the Euclidean case. By
developing suitable adaptations of the counting techniques, we derive
corresponding results in these two settings, thereby providing a unified
treatment of Euclidean, Hermitian, and symplectic hulls.

In \cite{LiShiLing_further}, the authors established enumeration formulae for linear
codes with prescribed Hermitian and symplectic hull dimensions and obtained
asymptotic results. Using these formulae, we prove results analogous
to Theorems \ref{thm:ineq} and \ref{thm:Ealpha}.

In the enumeration formulae, ${n\brack k}_q$ is the Gaussian binomial coefficient, defined by
\begin{equation}
{n\brack 0}_q=1, \ {n\brack k}_q=\frac{(q^n-1)(q^{n-1}-1)\ldots(q^{n-k+1}-1)}{(q^k-1)(q^{k-1}-1)\ldots(q-1)}, \mbox{if} \ 1\le k\le n.
\end{equation}

\section{Hermitian hulls}
\label{sect:Hermitian}

 Let $\mathbb{F}_{q^2}$ be a finite field with $q^2$ elements and $\mathbb{F}_{q^2}^n$ be the $n$-dimensional vector space over this field.
The following formula was proved in \cite[Theorem 3.11]{LiShiLing_further}.
Here $A_{n,k,\ell,q^2}^H$ denotes the number of all linear $[n,k]$ codes over
$\mathbb{F}_{q^2}$ whose Hermitian hull has dimension $\ell$.

\begin{theorem}\cite[Theorem 3.11]{LiShiLing_further}\label{hull}
Let $n$ and $k$ be positive integers and let $\ell$ be a non-negative integer
such that $\ell\le k\le n-\ell$. Set $k_0=k-\ell$ and
$D_{n,k_0}=\prod_{j=1}^{k_0}\frac{q^{n-k_0+j}-(-1)^{n-k_0+j}}{q^j-(-1)^j}$. Then
\[
A_{n,k,\ell,q^2}^H=\left\{\begin{array}{ll}
\left(\prod_{i=1}^\ell A_{n,k_0,i}\right)q^{k_0(n-k_0)}D_{n,k_0},
&\mbox{if} \ n-k_0 \ \mbox{is odd},\\
\left(\prod_{i=1}^\ell B_{n,k_0,i}\right)q^{k_0(n-k_0)}D_{n,k_0},
&\mbox{if} \ n-k_0 \ \mbox{is even},\\
\end{array}\right.
\]
where
$A_{n,k_0,i}=\frac{(q^{m_i+2}+1)(q^{m_i+1}-1)}{q^{2k_0}(q^{2i}-1)}$ and $B_{n,k_0,i}=\frac{(q^{m_i+2}-1)(q^{m_i+1}+1)}{q^{2k_0}(q^{2i}-1)}$,\\ $m_i=n-k_0-2i$.
\end{theorem}

When $\ell=0$, the following formula is also proved in \cite{LiShiLing_further}.

\begin{theorem}\cite[Theorem 3.4]{LiShiLing_further}\label{HLCD}
The number of Hermitian LCD $[n,k]_{q^2}$ codes is
\begin{equation}
A_{n,k,0,q^2}^H=q^{k(n-k)}\prod_{i=1}^k\frac{q^{n-k+i}-(-1)^{n-k+i}}{q^i-(-1)^i}.
\end{equation}
\end{theorem}

We use a more convenient form of the enumeration formula.

\begin{corollary}\label{cor:hhull} Let $n$ and $k$ be positive integers and let
$\ell$ be a non-negative integer such that $0\le \ell\le k\le n-\ell$.
Set $n_0=n-k$. Then
\begin{equation}\label{eq2}
A_{n,k,\ell,q^2}^H=
\left(\prod_{i=1}^\ell M_{n,k,\ell,i}^H\right)q^{(k-\ell)(n_0+\ell)}\prod_{i=1}^{k-\ell}\frac{q^{n_0+\ell+i}-(-1)^{n_0+\ell+i}}{q^i-(-1)^i},
\end{equation}
where $$M_{n,k,\ell,i}^H=\frac{(q^{n_0+\ell-2i+2}-(-1)^{n_0+\ell})(q^{n_0+\ell-2i+1}+(-1)^{n_0+\ell})}{q^{2k-2\ell}(q^{2i}-1)}.$$
\end{corollary}

\begin{proof} Obviously, $n_0+\ell=n-k+\ell=n-k_0$. Let $$M_{n,k,\ell,i}^H=\left\{\begin{array}{ll}
A_{n,k_0,i}, &\mbox{if} \ n-k_0=n_0+\ell \ \mbox{is odd},\\
B_{n,k_0,i}, &\mbox{if} \ n-k_0=n_0+\ell \ \mbox{is even}.\\
\end{array}\right.$$

Since
\begin{align*}
  A_{n,k_0,i} & =\frac{(q^{n-k_0-2i+2}+1)(q^{n-k_0-2i+1}-1)}{q^{2k_0}(q^{2i}-1)} \\
   & =\frac{(q^{n_0+\ell-2i+2}+1)(q^{n_0+\ell-2i+1}-1)}{q^{2k-2\ell}(q^{2i}-1)},
\end{align*}
and
\begin{align*}
  B_{n,k_0,i} & =\frac{(q^{n-k_0-2i+2}-1)(q^{n-k_0-2i+1}+1)}{q^{2k_0}(q^{2i}-1)} \\
   & =\frac{(q^{n_0+\ell-2i+2}-1)(q^{n_0+\ell-2i+1}+1)}{q^{2k-2\ell}(q^{2i}-1)},
\end{align*}
we have
$$M_{n,k,\ell,i}^H=\frac{(q^{n_0+\ell-2i+2}-(-1)^{n_0+\ell})(q^{n_0+\ell-2i+1}+(-1)^{n_0+\ell})}{q^{2k-2\ell}(q^{2i}-1)}.$$

Hence
\begin{align*}
  A_{n,k,\ell,q^2}^H &= \left(\prod_{i=1}^\ell M^H_{n,k,\ell,i}\right)q^{k_0(n-k_0)}\prod_{i=1}^{k_0}\frac{q^{n-k_0+i}-(-1)^{n-k_0+i}}{q^i-(-1)^i}\\
  &= \left(\prod_{i=1}^\ell M^H_{n,k,\ell,i}\right)q^{(k-\ell)(n_0+\ell)}\prod_{i=1}^{k-\ell}\frac{q^{n_0+\ell+i}-(-1)^{n_0+\ell+i}}{q^i-(-1)^i}.
\end{align*}

In the case $\ell=0$, we have $\prod_{i=1}^0 M_{n,k,0,i}^H=1$, and therefore 
formula~\eqref{eq2} holds in this case as well.
\end{proof}

In the rest of this section we build our investigation on Corollary \ref{cor:hhull} and Theorem \ref{HLCD}. 

\begin{theorem}\label{main_hermitian}
Let $n$ and $k$ be positive integers and let $\ell$ be a non-negative integer
with $\ell+1\le k\le n-\ell-1$. Then
\begin{equation}
 A_{n,k,\ell,q^2}^H= \alpha^H_{n,k,\ell,q^2}(q^{\ell+1}-1)  A_{n,k,\ell+1,q^2}^H,
\end{equation}
where
\begin{equation}
\alpha^H_{n,k,\ell,q^2}=\frac{q^{n-2\ell-1}(q^{\ell+1}+1)} {(q^{n-k-\ell}-(-1)^{n-k-\ell})(q^{k-\ell}-(-1)^{k-\ell})}.
\end{equation}
Moreover, 
\begin{equation}
\alpha^H_{n,k,\ell,q^2}\left\{\begin{array}{ll}
>(q^{\ell+1}+1)/q& \mbox{if} \ k-\ell \ \mbox{and} \ n \ \mbox{are even}\\
>(q^{\ell+1}+1)/(q+1)& \mbox{if} \ n \ \mbox{is odd}\\
\ge q(q^{\ell+1}+1)/(q+1)^2& \mbox{if} \ k-\ell \ \mbox{and} \ n-k-\ell \ \mbox{are odd}.
\end{array}\right.
\end{equation}
In all cases $\alpha^H_{n,k,\ell,q^2}> 1$ except when $\ell=0$, $n$ is even and $k\in\{1,n-1\}$; in that exceptional case
$\alpha^H_{n,k,0,q^2}\ge\frac{q}{q+1}\ge\frac{2}{3}$.
\end{theorem}

\begin{proof} Using the formulae from Corollary \ref{cor:hhull}, for $1\le \ell\le k-1$ we have
$$\frac{\prod_{i=1}^\ell M_{n,k,\ell,i}^H}{\prod_{i=1}^\ell M_{n,k,\ell+1,i}^H}=q^{-2\ell}\frac{q^{n_0-\ell+1}+(-1)^{n_0+\ell}}{q^{n_0+\ell+1}+(-1)^{n_0+\ell}}, \ \mbox{and}$$
$$\prod_{i=1}^{k-\ell-1}\frac{q^{n_0+\ell+i}-(-1)^{n_0+\ell+i}}{q^{n_0+\ell+i+1}-(-1)^{n_0+\ell+i+1}}=\frac{q^{n_0+\ell+1}+(-1)^{n_0+\ell}}{q^{n}-(-1)^{n}}.$$

This gives us that
\begin{align*}
  \frac{A_{n,k,\ell,q^2}^H}{ A_{n,k,\ell+1,q^2}^H} & =\left(\frac{\prod_{i=1}^\ell M_{n,k,\ell,i}^H}{\prod_{i=1}^{\ell+1} M_{n,k,\ell+1,i}^H}\right)\frac{q^{(k-\ell)(n_0+\ell)}\prod_{i=1}^{k-\ell}\frac{q^{n_0+\ell+i}-(-1)^{n_0+\ell+i}}{q^i-(-1)^i}}
  {q^{(k-\ell-1)(n_0+\ell+1)}\prod_{i=1}^{k-\ell-1}\frac{q^{n_0+\ell+i+1}-(-1)^{n_0+\ell+i+1}}{q^i-(-1)^i}} \\
  &=q^{-2\ell}\frac{q^{n_0-\ell+1}+(-1)^{n_0+\ell}}{q^{n_0+\ell+1}+(-1)^{n_0+\ell}}\cdot\frac{q^{n_0+\ell+1}+(-1)^{n_0+\ell}}{q^{n}-(-1)^{n}}\\
  &\quad\cdot\frac{q^{n_0-k+2\ell+1}(q^{n}-(-1)^{n})}{M_{n,k,\ell+1,\ell+1}^H(q^{k-\ell}-(-1)^{k-\ell})}\\
  &=\frac{q^{n_0-k+1}(q^{n_0-\ell+1}+(-1)^{n_0+\ell})q^{2k-2\ell-2}(q^{2\ell+2}-1)}{(q^{k-\ell}-(-1)^{k-\ell})(q^{n_0-\ell+1}+(-1)^{n_0+\ell})(q^{n_0-\ell}-(-1)^{n_0+\ell})}\\
  &=\frac{q^{n-2\ell-1}(q^{2\ell+2}-1)}{(q^{k-\ell}-(-1)^{k-\ell})(q^{n_0-\ell}-(-1)^{n_0+\ell})}.
\end{align*}

In the case $\ell=0$ we obtain the same formula, namely
$$\frac{A_{n,k,0,q^2}^H}{ A_{n,k,1,q^2}^H}=\frac{q^{n-1}(q^{2}-1)}{(q^{k}-(-1)^{k})(q^{n-k}-(-1)^{n-k})}.$$

Hence $$\alpha^H_{n,k,\ell,q^2}=\frac{q^{n-2\ell-1}(q^{\ell+1}+1)} {(q^{n-k-\ell}-(-1)^{n-k-\ell})(q^{k-\ell}-(-1)^{k-\ell})}=\frac{q^{\ell+1}+1} {q(1-\frac{(-1)^{n-k-\ell}}{q^{n-k-\ell}})(1-\frac{(-1)^{k-\ell}}{q^{k-\ell}})}.$$

To prove the inequality in the different cases, we use that
$$1-\frac{1}{q^m}<1  \ \mbox{for} \ m> 0, \ \mbox{and} \ 1+\frac{1}{q^m}< 1+\frac{1}{q} \ \mbox{for} \ m> 1.$$
Since $\ell+1\le k\le n-\ell-1$, it turns out that $k-\ell\ge 1$ and $n-k-\ell\ge 1$. We have $k-\ell=n-k-\ell=1$ only when $\ell=k-1$ and $n=2k$.

It follows that
\begin{itemize}
\item $\alpha^H_{n,k,\ell,q^2}>(q^{\ell+1}+1)/q>1$ if $k-\ell$ and $n$ are even;
\item $\alpha^H_{n,k,\ell,q^2}>(q^{\ell+1}+1)/(q+1)\ge 1$ if $k-\ell$ and $n-k-\ell$ are of different parity;
\item  $\alpha^H_{n,k,\ell,q^2}\ge q(q^{\ell+1}+1)/(q+1)^2$ when both $k-\ell$ and $n-k-\ell$ are odd.
\end{itemize} 
Let us consider the last case in greater detail, when $k-\ell$ is odd and
$n$ is even. If $\ell\ge 1$, then
$$\frac{q(q^{\ell+1}+1)}{(q+1)^2}\ge \frac{q(q^2+1)}{(q+1)^2}>1.$$

Let now $\ell=0$. If $k\ge 3$ and $n-k\ge 3$ are odd integers, then
$$\alpha^H_{n,k,0,q^2}\ge\frac{q^5(q+1)} {(q^3+1)^2}>1.$$

If $k=1$ or $k=n-1$, and $n$ is even, then $\alpha^H_{n,k,0,q^2}=\frac{q^{n-1}}{q^{n-1}+1}\ge \frac{q}{q+1}\ge 2/3.$
   \end{proof}
   
   \begin{corollary}\label{cor:Hstrong}
For all $q \ge 2$ and $\ell+1 \le k \le n-\ell-1$,
\[
A_{n,k,\ell,q^2}^H>
\bigl(q^{\ell+1}-1\bigr)\,
A_{n,k,\ell+1,q^2}^H
\]
holds for all valid parameters except when $\ell=0$, $n$ is even, and
$k\in\{1,n-1\}$. In each such exceptional case,
\[
A_{n,k,0,q^2}^H= \frac{q^{\,n-1}(q-1)}{q^{\,n-1}+1}\,
A_{n,k,1,q^2}^H.
\]
If $q\ge 3$, then $A_{n,k,0,q^2}^H> A_{n,k,1,q^2}^H$, but if
$q=2$ then \[
A_{n,k,0,4}^H= \frac{2^{\,n-1}}{2^{\,n-1}+1}\,
A_{n,k,1,4}^H<A_{n,k,1,4}^H.
\]
\end{corollary}

\begin{remark} In fact, if $n$ is even, then
$$A_{n,1,0,4}^H=\frac{2^{n-1}(2^n-1)}{3}, \ \ A_{n,1,1,4}^H=\frac{(2^{n-1}+1)(2^n-1)}{3}.$$
If $n$ is odd then
$$A_{n,1,0,4}^H=\frac{2^{n-1}(2^n+1)}{3}, \ \ A_{n,1,1,4}^H=\frac{(2^{n-1}-1)(2^n+1)}{3}.$$
\end{remark}
   
\begin{corollary}[Asymptotic behavior of $A^H_\ell/A^H_{\ell+1}$]\label{cor:alpha-herm-asym}
Let $\ell\ge0$ and $q$ be fixed, and set $a=k-\ell$, $b=n-k-\ell$
with $a,b\ge1$ (so $a+b+2\ell=n$).
\begin{enumerate}
\item (\emph{Boundary regime: $k$ fixed, $n\to\infty$.})
Fix $a\ge1$ and let $b\to\infty$. Then
\[
  \frac{A^H_{n,k,\ell,q^2}}{A^H_{n,k,\ell+1,q^2}}
  \longrightarrow \frac{q^{a-1}(q^{\ell+1}+1)(q^{\ell+1}-1)}{q^a-(-1)^a}.
\]
\item (\emph{Joint regime: $k\to\infty$ and $n-k\to\infty$.})
Let $a,b\to\infty$ simultaneously. Then
\[
  \frac{A^H_{n,k,\ell,q^2}}{A^H_{n,k,\ell+1,q^2}}
  \longrightarrow \frac{q^{2(\ell+1)}-1}{q}.
\]
\item The asymptotic ratio in~(2) is minimized at $\ell=0$, equal to
$(q^2-1)/q=q-q^{-1}$, which is $3/2$ at $q=2$ and grows like $q$; hence
$A^H_{n,k,\ell,q^2}>A^H_{n,k,\ell+1,q^2}$ asymptotically for every $\ell\ge0$
and every $q\ge2$.
\end{enumerate}
\end{corollary}

\begin{proof}
By Theorem~\ref{main_hermitian},
$A^H_{n,k,\ell,q^2}/A^H_{n,k,\ell+1,q^2}=\alpha^H_{n,k,\ell,q^2}(q^{\ell+1}-1)$,
and using $a+b+2\ell=n$ the closed form of $\alpha^H_{n,k,\ell,q^2}$ can be
rewritten as
$\alpha^H=q^{a+b-1}(q^{\ell+1}+1)/[(q^a-(-1)^a)(q^b-(-1)^b)]$.

(1) For fixed $a$ and $b\to\infty$, $(q^b-(-1)^b)/q^b\to1$, so
\[
  \alpha^H=\frac{q^{a-1}(q^{\ell+1}+1)}{q^a-(-1)^a}\cdot\frac{q^b}{q^b-(-1)^b}
  \longrightarrow \frac{q^{a-1}(q^{\ell+1}+1)}{q^a-(-1)^a},
\]
and multiplying by $(q^{\ell+1}-1)$ gives the displayed formula.

(2) For $a,b\to\infty$ jointly, both $(q^a-(-1)^a)/q^a$ and $(q^b-(-1)^b)/q^b$
tend to $1$, hence $\alpha^H\to(q^{\ell+1}+1)/q$ and
$A^H_{n,k,\ell,q^2}/A^H_{n,k,\ell+1,q^2}\to(q^{\ell+1}+1)(q^{\ell+1}-1)/q=(q^{2(\ell+1)}-1)/q$.

(3) The map $\ell\mapsto(q^{2(\ell+1)}-1)/q$ is strictly increasing on
$\ell\ge0$, so the minimum is attained at $\ell=0$ with value $(q^2-1)/q$.
At $q=2$ this equals $3/2>1$, and it increases without bound in $q$, so the
ratio strictly exceeds $1$ for all $q\ge2$ and all $\ell\ge0$.
\end{proof}

Table~\ref{tab:hermitian-A} lists $A^H_{n,k,\ell,q^2}$ from Theorem~\ref{hull}
for small parameters at $q=2$ (codes over $\mathbb{F}_4$) and $q=3$ (codes over
$\mathbb{F}_9$). Boldface marks entries where $A^H_{\ell}$ strictly exceeds
$A^H_{\ell-1}$, confirming Corollary~\ref{cor:Hstrong} on the boundary
exception family $\{(n,k,0):n\text{ even},\ k\in\{1,n-1\}\}$ at $q=2$.

\begin{table}[h]
\centering
\caption{$A^H_{n,k,\ell,q^2}$ (closed form, Theorem~\ref{hull}).
Boldface marks $A^H_{\ell+1}>A^H_{\ell}$ (boundary exception of
Corollary~\ref{cor:Hstrong}).}
\label{tab:hermitian-A}
\begin{tabular}{ccc|rrrr}
\toprule
$n$ & $k$ & $q$ & $A^H_0$ & $A^H_1$ & $A^H_2$ & $A^H_3$ \\
\midrule
4 & 1 & 2 & 40 & \textbf{45} &  & \\
5 & 1 & 2 & 176 & 165 &  & \\
6 & 1 & 2 & 672 & \textbf{693} &  & \\
7 & 1 & 2 & 2752 & 2709 &  & \\
4 & 2 & 2 & 240 & 90 & 27 & \\
5 & 2 & 2 & 3520 & 1980 & 297 & \\
6 & 2 & 2 & 59136 & 27720 & 6237 & \\
6 & 3 & 2 & 197120 & 166320 & 12474 & 891 \\
7 & 2 & 2 & 924672 & 476784 & 89397 & \\
7 & 3 & 2 & 13561856 & 9535680 & 1072764 & 38313 \\
8 & 2 & 2 & 14970880 & 7368480 & 1519749 & \\
\midrule
4 & 1 & 3 & 540 & 280 &  & \\
5 & 1 & 3 & 4941 & 2440 &  & \\
6 & 1 & 3 & 44226 & 22204 &  & \\
4 & 2 & 3 & 5670 & 1680 & 112 & \\
5 & 2 & 3 & 444690 & 153720 & 6832 & \\
6 & 2 & 3 & 36420111 & 11990160 & 621712 & \\
6 & 3 & 3 & 312172380 & 125896680 & 3730272 & 27328 \\
\bottomrule
\end{tabular}
\end{table}

At $q=2$, the bold entries at $[4,1]_4$ and $[6,1]_4$ confirm the
$A$-monotonicity boundary failure: precisely the family
$(q,\ell,k)=(2,0,1)$ or $(2,0,n-1)$ with $n$ even. At $q=3$ the analogous
parameters $[4,1]_9,[5,1]_9,[6,1]_9$ all satisfy $A^H_0>A^H_1$, consistent
with Corollary~\ref{cor:Hstrong} restricting the failure to $q=2$.


\section{Symplectic hulls}
\label{sect:symplectic}

We now establish the analogous theory for the symplectic inner product. The
structure parallels Section \ref{sect:Hermitian}, but with two key differences:
the necessary condition that $k-\ell$ is even
\cite[Proposition 4.4]{LiShiLing_further} forces admissible hull dimensions
to change in steps of two, and the closed-form ratio $\alpha^S$
behaves qualitatively differently from $\alpha^H$.


%
%

\begin{theorem}\cite[Theorem 4.9]{LiShiLing_further}
Assume that $n$, $k$ and $\ell$ are non-negative integers with
$\ell\le k\le 2n-\ell$ and $k=2k_0+\ell$ for a non-negative integer $k_0$. Then
\begin{equation}
A_{2n,k,\ell,q}^S=\left(\prod_{i=1}^\ell\frac{q^{2n-2k_0-2i+2}-1}{q^{2k_0+i}-q^{2k_0}}\right)q^{2k_0(n-k_0)}{n\brack k_0}_{q^2}.
\end{equation}
\end{theorem}

In the above theorem, $\displaystyle{n\brack k_0}_{q^2}$ is the Gaussian binomial coefficient. 

\begin{remark} The formula appearing in the statement of Theorem 4.9 in \cite{LiShiLing_further} contains an error, although the proof itself is correct. Here we present the corrected formula.
\end{remark}

Again, as in Section \ref{sect:Hermitian}, we use a more convenient formula presented in the following corollary.

\begin{corollary}\label{cor:SHULL}
Assume that $n$, $k$ and $\ell$ are non-negative integers with
$\ell\le k\le 2n-\ell$ and $k=2k_0+\ell$ for a non-negative integer $k_0$. If $\ell\ge 1$ then
\begin{equation}
A_{2n,k,\ell,q}^S=q^{2k_0(n-k_0-\ell)}{n\brack k_0}_{q^2}\prod_{i=1}^\ell\frac{q^{2n-2k_0-2i+2}-1}{q^{i}-1}.
\end{equation}
Furthermore, if $\ell=0$ (so $k=2k_0$), then
\begin{equation}
A_{2n,k,0,q}^S=q^{2k_0(n-k_0)}{n\brack k_0}_{q^2}.
\end{equation}
\end{corollary}

In the rest of this section, as in Section \ref{sect:Hermitian}, we build our investigation on Corollary \ref{cor:SHULL}. Note that Corollary \ref{cor:SHULL} also covers the case $\ell=0$.

Note that a $[2n,k]_q$ code with symplectic hull dimension $\ell$ is a subspace of type $(k,(k-\ell)/2)$ of the $2n$-dimensional symplectic space over $\F_q$ in the terminology of Wan \cite{Wan}, and the formula in Corollary \ref{cor:SHULL} agrees with the number of such subspaces given in \cite[Theorem 3.18]{Wan}.

\begin{theorem}\label{main_symplectic}
Let $n$, $k$, $\ell$ be non-negative integers with $2\le\ell+2\le k\le 2n-\ell-2$ and $k-\ell$ even. Then
\begin{equation}
 A_{2n,k,\ell,q}^S= \alpha^S_{2n,k,\ell,q}(q^{\ell+1}-1)(q^{\ell+2}-1) A_{2n,k,\ell+2,q}^S,
\end{equation}
where
\begin{equation}
\alpha^S_{2n,k,\ell,q}=\frac{q^{2n-2\ell-2}}{(q^{k-\ell}-1)(q^{2n-k-\ell}-1)}>\frac{1}{q^2}.
\end{equation}
In all cases $\alpha^S_{2n,k,\ell,q}(q^{\ell+2}-1)> 1$ except when $\ell=0$ and $4\le k\le 2n-4$; in that exceptional case
$\alpha^S_{2n,k,\ell,q}(q^{\ell+2}-1)\ge\frac{q^2-1}{q^2}$.
\end{theorem}

\begin{proof}
As in the previous section, we evaluate the ratio $A_{2n,k,\ell,q}^S/  A_{2n,k,\ell+2,q}^S$ as a product of four factors arising from the four ingredients of the formula:
\begin{itemize}
\item The prefactor ratio gives $q^{2n-2\ell-2}$.
\item The Gaussian binomial ratio gives $\frac{q^{2(n-k_0+1)}-1}{q^{2k_0}-1}$.
\item The numerator product ratio gives $1/[(q^{2n-2k_0+2}-1)(q^{2n-2k_0-2\ell}-1)]$.
\item The denominator product ratio gives $(q^{\ell+1}-1)(q^{\ell+2}-1) $.
\end{itemize}

Multiplying these together, we obtain
\[
  \frac{A_{2n,k,\ell,q}^S}{A_{2n,k,\ell+2,q}^S}  =\frac{q^{2n-2\ell-2}(q^{\ell+1}-1)(q^{\ell+2}-1)}{(q^{2k_0}-1)(q^{2n-2k_0-2\ell}-1)}
  \]
  
Hence $$\alpha^S_{2n,k,\ell,q}=\frac{q^{2n-2\ell-2}}{(q^{2k_0}-1)(q^{2n-2k_0-2\ell}-1)}=\frac{q^{2n-2\ell-2}}{(q^{k-\ell}-1)(q^{2n-k-\ell}-1)}>\frac{1}{q^2}.$$

Consider now $\alpha^S_{2n,k,\ell,q}(q^{\ell+2}-1)$. If $\ell\ge 1$ then
$$\alpha^S_{2n,k,\ell,q}(q^{\ell+2}-1)>\frac{q^{\ell+2}-1}{q^2}\ge\frac{q^3-1}{q^2}>1.$$

Let $\ell=0$. In this case $k$ must be even, and  $$\alpha^S_{2n,k,\ell,q}(q^{2}-1)=\frac{q^{2n-2}(q^2-1)}{(q^{k}-1)(q^{2n-k}-1)}.$$
If $k=2$ or $k=2n-2$ then $$\alpha^S_{2n,k,\ell,q}(q^{2}-1)=\frac{q^{2n-2}}{q^{2n-2}-1}>1.$$
In the remaining case, when $4\le k\le 2n-4$, we have $q^k+q^{2n-k}\le 2q^{2n-4}<q^{2n-2}$. It turns out that in this case
$$\alpha^S_{2n,k,\ell,q}(q^{2}-1)=\frac{q^{2n}-q^{2n-2}}{q^{2n}-q^k-q^{2n-k}+1}<\frac{q^{2n}-q^{2n-2}}{q^{2n}-q^{2n-2}+1}<1,$$
but $\alpha^S_{2n,k,\ell,q}(q^{2}-1)>(q^2-1)/q^2$.
\end{proof}

This proves the following theorem

\begin{theorem}\label{thm:SHull}
Assume that $n$, $k$ and $\ell$ are three nonnegative integers with $\ell+2\le k\le 2n-\ell-2$, and $k-\ell$ is even. Then
\[
A_{2n,k,\ell,q}^S>(q^{\ell+1}-1)A_{2n,k,\ell+2,q}^S,\]
except when $\ell=0$ and $4\le k\le 2n-4$. If $\ell=0$ and $q\ge 3$, then $A_{2n,k,\ell,q}^S>A_{2n,k,\ell+2,q}^S$. If $\ell=0$, $q=2$ and $4\le k\le 2n-4$ then
$$\frac{3}{4}A^S_{2n,k,2,2}<A^S_{2n,k,0,2}<A^S_{2n,k,2,2}.$$
\end{theorem}

\begin{corollary}[Asymptotic behavior of $A^S_\ell/A^S_{\ell+2}$]\label{cor:alpha-symp-asym}
Let $\ell\ge0$ and $q$ be fixed, and set $a=k-\ell$, $b=2n-k-\ell$
with $a,b\ge2$ even (so $a+b+2\ell=2n$).
\begin{enumerate}
\item (\emph{Boundary regime: $k$ fixed, $n\to\infty$.})
Fix $a\ge2$ and let $b\to\infty$. Then
\[
  \frac{A^S_{2n,k,\ell,q}}{A^S_{2n,k,\ell+2,q}}
  \longrightarrow \frac{q^{a-2}(q^{\ell+1}-1)(q^{\ell+2}-1)}{q^a-1}.
\]
\item (\emph{Joint regime: $k\to\infty$ and $2n-k\to\infty$.})
Let $a,b\to\infty$ simultaneously. Then
\[
  \frac{A^S_{2n,k,\ell,q}}{A^S_{2n,k,\ell+2,q}}
  \longrightarrow \frac{(q^{\ell+1}-1)(q^{\ell+2}-1)}{q^2}.
\]
\item The asymptotic ratio in~(2) is minimized at $\ell=0$, equal to
$(q-1)(q^2-1)/q^2$. At $q=2$ this equals $3/4<1$: the joint limit at $\ell=0$,
$q=2$ is taken within the exception family ($q=2$, $\ell=0$,
$4\le k\le 2n-4$) of Theorem~\ref{thm:SHull}, which contains arbitrarily large
parameters, so $3/4<1$ confirms that the BBO-type inequality fails throughout
this family as $2n$ grows. For $q\ge3$ the asymptotic ratio is
$(q-1)^2(q+1)/q^2\ge16/9>1$. For each fixed $\ell\ge1$ the asymptotic ratio is
strictly larger (e.g.\ $A^S_1/A^S_3\to21/4$ at $q=2$), so monotonicity holds
asymptotically in all remaining cases.
\end{enumerate}
\end{corollary}

\begin{proof}
By Theorem~\ref{main_symplectic},
$A^S_{2n,k,\ell,q}/A^S_{2n,k,\ell+2,q}=\alpha^S_{2n,k,\ell,q}(q^{\ell+1}-1)(q^{\ell+2}-1)$
with $\alpha^S=q^{a+b-2}/[(q^a-1)(q^b-1)]$.

(1) For fixed $a$ and $b\to\infty$, $(q^b-1)/q^b\to1$, so
\[
  \alpha^S=\frac{q^{a-2}}{q^a-1}\cdot\frac{q^b}{q^b-1}\longrightarrow\frac{q^{a-2}}{q^a-1},
\]
and multiplying by $(q^{\ell+1}-1)(q^{\ell+2}-1)$ gives the displayed formula.

(2) For $a,b\to\infty$ jointly, both $(q^a-1)/q^a$ and $(q^b-1)/q^b$ tend to
$1$, hence $\alpha^S\to q^{-2}$ and
$A^S_{2n,k,\ell,q}/A^S_{2n,k,\ell+2,q}\to(q^{\ell+1}-1)(q^{\ell+2}-1)/q^2$.

(3) The function $\ell\mapsto(q^{\ell+1}-1)(q^{\ell+2}-1)$ is strictly increasing
on $\ell\ge0$, so the minimum over admissible $\ell$ is attained at $\ell=0$
with value $(q-1)(q^2-1)$. At $q=2$ this divides by $q^2=4$ to give $3/4$. For
$q\ge3$, $(q-1)^2(q+1)/q^2$ is increasing in $q$ and equals $16/9$ at $q=3$, so
the quantity is $\ge16/9$ throughout. At $\ell\ge1$, the asymptotic ratio
is at least $(q^2-1)(q^3-1)/q^2\ge21/4>1$, with equality in the first
inequality when $\ell=1$ and in the second when $q=2$, hence exceeds $1$
for all $\ell\ge1$ and all $q\ge2$.
\end{proof}

Table~\ref{tab:symplectic-A} lists $A^S_{2n,k,\ell,q}$ from
Corollary~\ref{cor:SHULL} for small parameters. Boldface marks entries where
$A^S_\ell$ strictly exceeds $A^S_{\ell-2}$, confirming Theorem~\ref{thm:SHull}
on the exception family ($q=2$, $\ell=0$, $4\le k\le 2n-4$, $k$ even) of
$A$-monotonicity failures.

\begin{table}[h]
\centering
\caption{$A^S_{2n,k,\ell,q}$ (closed form, Corollary~\ref{cor:SHULL}).
Boldface marks $A^S_2>A^S_0$.}
\label{tab:symplectic-A}
\begin{tabular}{ccc|rrrr}
\toprule
$2n$ & $k$ & $q$ & $A^S_0$ & $A^S_2$ & $A^S_4$ & $A^S_6$ \\
\midrule
4 & 2 & 2 & 20 & 15 & & \\
6 & 2 & 2 & 336 & 315 & & \\
8 & 2 & 2 & 5440 & 5355 & & \\
8 & 4 & 2 & 91392 & \textbf{107100} & 2295 & \\
10 & 2 & 2 & 87296 & 86955 & & \\
10 & 4 & 2 & 23744512 & \textbf{29216880} & 782595 & \\
12 & 4 & 2 & 6100942848 & \textbf{7596388800} & 213648435 & \\
12 & 6 & 2 & 98777169920 & \textbf{127619331840} & 4272968700 & 4922775 \\
\midrule
4 & 2 & 3 & 90 & 40 & & \\
6 & 2 & 3 & 7371 & 3640 & & \\
8 & 2 & 3 & 597780 & 298480 & & \\
8 & 4 & 3 & 48958182 & 26863200 & 91840 & \\
\bottomrule
\end{tabular}
\end{table}

At $q=2$, all entries with $4\le k\le 2n-4$ exhibit $A^S_0<A^S_2$, consistent
with Theorem~\ref{thm:SHull}. At $q=3$, $A^S$ is strictly monotone in $\ell$ in
every computed case, consistent with Theorem~\ref{thm:SHull} restricting
failures to $q=2$.

\section{Comparison and applications}
\label{sect:comparison}

Table~\ref{tab:comparison} juxtaposes the Euclidean (BBO~\cite{BBO2025}),
Hermitian (Section~\ref{sect:Hermitian}), and symplectic
(Section~\ref{sect:symplectic}) ratio decompositions side by side.

\begin{table}[h]
\centering
\footnotesize
\caption{Comparison of the ratio decompositions across the three classical
inner products. Here $a = k - \ell$, $b = n - k - \ell$ (Euclidean/Hermitian)
or $b = 2n - k - \ell$ (symplectic); $\Delta$ is the step in $\ell$. The asymptotic rows have fixed $q$, $\ell=0$ and $a,b\to\infty$. The last row gives the exceptions to $A_\ell>(q^{\ell+1}-1)A_{\ell+\Delta}$ with positive adjacent counts.}
\label{tab:comparison}
\setlength{\tabcolsep}{2pt}
\begin{tabular}{l|c|c|c}
\toprule
& Euclidean~\cite{BBO2025} & Hermitian & Symplectic \\
\midrule
Step $\Delta$ in $\ell$ & $1$ & $1$ & $2$ \\[2pt]
$\alpha^*$ closed form &
   case-split &
   $\dfrac{q^{n-2\ell-1}(q^{\ell+1}+1)}{(q^a-(-1)^a)(q^b-(-1)^b)}$ &
   $\dfrac{q^{a+b-2}}{(q^a-1)(q^b-1)}$ \\[6pt]
$\alpha^*$ lower bound & $1/2$ & $q/(q+1) \geq 2/3$ & $>1/q^2$ \\
$\alpha^*$ asymptotic & $1$ & $(q+1)/q$ & $1/q^2$ \\[2pt]
Asymp.\ $A_0/A_\Delta$ & $q{-}1$ & $(q^2{-}1)/q$ & $(q{-}1)(q^2{-}1)/q^2$ \\
\;at $q = 2$ & $1$ & $3/2$ & $\mathbf{3/4}$ \\
\;at $q \geq 3$ & $\geq 2$ & $\geq 8/3$ & $\geq 16/9$ \\
\midrule
Exceptions & \shortstack{$q$ odd, $n$ even\\$\eta((-1)^{n/2})=1$\\$k-\ell$ odd} & $\ell{=}0,\,n$ even, $k{\in}\{1,n{-}1\}$ & $\ell{=}0,\,4{\leq}k{\leq}2n{-}4$ \\
\bottomrule
\end{tabular}
\end{table}

The contrast in Table~\ref{tab:comparison} tracks the Witt classification
of the underlying form. The orthogonal group over odd $\F_q$ splits
non-degenerate subspaces into two isometry classes (the discriminant lying in
$\F_q^*/(\F_q^*)^2$), producing the $\eta$-character case-split in the BBO
formula and the worst-case bound $\alpha^E \geq 1/2$. The unitary and symplectic
groups, in contrast, each act with a single orbit on LCD codes, so the closed
forms $\alpha^H$ and $\alpha^S$ contain no quadratic-character factor. The
symplectic case differs: every vector is self-orthogonal and the
parity constraint $k - \ell \in 2\mathbb{Z}$ is forced. At $\ell=0$, the
symplectic ratio factor tends to $1/q^2$, whereas the Hermitian ratio factor
tends to $(q+1)/q$. The two exception sets to the BBO-style inequality are
also different: the boundary family $\ell=0$, $n$ even, $k\in\{1,n-1\}$ in the Hermitian case versus the family $\ell=0$, $4\le k\le2n-4$ in the symplectic case.

Next, we compare the numbers of codes with the same length, dimension, and hull dimension under the different inner products considered in this study.
For this purpose, we employ a more convenient expression for $A_{n,k,\ell,q}^E$.

\begin{theorem}\label{thmE} Let $n$ and $k$ be positive integers and let $\ell$ be a non-negative integer such that $\ell\le k<n$ and $k\le n-\ell$. Suppose that $k_0=k-\ell$. Then
\begin{equation}
A_{n,k,\ell,q}^E=\left\{\begin{array}{ll}
\left(\prod_{i=1}^\ell M^E_{n,k,\ell,q,i}\right)q^{(k_0(n-k_0)-1)/2}B_1{n/2-1\brack(k_0-1)/2 }_{q^2}&\mbox{if} \ n \ \mbox{is even}, k_0 \ \mbox{is odd},\\
\left(\prod_{i=1}^{\ell} M^E_{n,k,\ell,q,i}\right)q^{(k_0(n-k_0)+n-k-\ell)/2}{(n-1)/2\brack(k_0-1)/2 }_{q^2}&\mbox{if} \ n \ \mbox{is odd}, k_0 \ \mbox{is odd},\\
\left(\prod_{i=1}^{\ell} M^E_{n,k,\ell,q,i}\right)q^{(k_0(n-k_0)+k-\ell)/2}{(n-1)/2\brack k_0/2 }_{q^2}&\mbox{if} \ n \ \mbox{is odd}, k_0 \ \mbox{is even},\\
\left(\prod_{i=1}^\ell M^E_{n,k,\ell,q,i}\right)q^{k_0(n-k_0)/2}B_2{n/2-1\brack k_0/2 }_{q^2}&\mbox{if} \ n \ \mbox{is even}, k_0 \ \mbox{is even}.\\
\end{array}\right.
\end{equation}
where 
$$M^E_{n,k,\ell,q,i}=\left\{\begin{array}{ll}
\frac{q^{n-k_0-2i+1}-1}{q^{k_0}(q^i-1)} &\mbox{if} \ n-k_0 \ \mbox{is odd,}\\
\frac{q^{n-k_0-2i+2}-1}{q^{k_0}(q^i-1)}& \mbox{if} \ n-k_0 \ \mbox{is even},
\end{array}\right.$$
(i) $B_1=q^{n/2}-\eta((-1)^{n/2})$ and $B_2=\frac{(q^{n/2-\ell}+\eta((-1)^{n/2}))(q^{n}-1)}{(q^{n/2}+\eta((-1)^{n/2}))(q^{n-k_0}-1)}$, if $q$ is odd, and (ii) $B_1=q^{n/2}$ and $B_2=\frac{q^{n-\ell}-1}{q^{n-k_0}-1}$, if $q$ is even.
\end{theorem}

\subsection{Euclidean vs symplectic hulls}

Using Corollary~\ref{cor:SHULL} with $\ell=k$, we calculate the number of symplectic self-orthogonal $[2n,k]_q$ codes. This number is $\pi_{2n,k,q}$, where the function $\pi_{2n,k,q}=\prod_{i=1}^k\frac{q^{2n-2i+2}-1}{q^i-1}$ is used in the formula for the number of Euclidean self-orthogonal codes $\sigma_{2n,k,q}$ proved by Vera Pless in \cite{Pless_sigma}, namely
\begin{equation*}
  \sigma_{2n,k,q} = \left\{ \begin{array}{rl}
\frac{q^{2n-k}-1}{q^{2n}-1}\pi_{2n,k,q},&\mbox{if} \ q \ \text{is even}, k\le n\\
\frac{q^{n-k}+1}{q^n+1}\pi_{2n,k,q},&\mbox{if} \ q \ \text{is odd} \ \mbox{and} \ (-1)^n \ \text{is a square}, k\le n\\
\frac{q^{n-k}-1}{q^n-1}\pi_{2n,k,q},&\mbox{if} \ q \ \text{is odd} \ \mbox{and} \ (-1)^n \ \text{is not a square}, k\le n-1.
\end{array}\right.
\end{equation*}

These formulas show that the symplectic self-orthogonal $[2n,k]_q$ codes are much more than the Euclidean self-orthogonal codes with the same parameters. For the self-dual codes in even $q$ we have $\pi_{2n,n,q}=(q^n+1)\sigma_{2n,n}$.

This observation leads to the following question: for which values of $\ell$ do we have $A_{2n,k,\ell,q}^S>A_{2n,k,\ell,q}^E$? We need to consider the formula for $A_{2n,k,\ell,q}^E$ only when $k-\ell=2k_0$ is even. In this case
\begin{align*}
A_{2n,2k_0+\ell,\ell,q}^E &=\left(\prod_{i=1}^\ell M^E_{2n,2k_0+\ell,\ell,q,i}\right)q^{2k_0(n-k_0)}B_2{n-1\brack k_0}_{q^2}\\
&=\left(\prod_{i=1}^\ell \frac{q^{2n-2k_0-2i+2}-1}{q^{2k_0}(q^i-1)}\right)q^{2k_0(n-k_0)}B_2{n-1\brack k_0}_{q^2}\\
&=\frac{A_{2n,k,\ell,q}^S}{{n\brack k_0}_{q^2}}B_2{n-1\brack k_0}_{q^2}=\frac{(q^{2(n-k_0)}-1)B_2}{q^{2n}-1}A_{2n,k,\ell,q}^S.
\end{align*}

If $q=2^s$, $s\ge 1$, then $B_2=\frac{q^{2n-\ell}-1}{q^{2n-2k_0}-1}$ and so 
$$A_{2n,2k_0+\ell,\ell,q}^E=\frac{q^{2n-\ell}-1}{q^{2n}-1}A_{2n,2k_0+\ell,\ell,q}^S.$$
It follows that
$A_{2n,2k_0,0,2^s}^E=A_{2n,2k_0,0,2^s}^S$. If $\ell\ge 1$, then $A_{2n,k,\ell,2^s}^E<A_{2n,k,\ell,2^s}^S$.

In the case when $q$ is odd, we have $B_2=\frac{(q^{n-\ell}+\eta((-1)^{n}))(q^{2n}-1)}{(q^{n}+\eta((-1)^{n}))(q^{2n-2k_0}-1)}$. Now
$$A_{2n,2k_0+\ell,\ell,q}^E=\frac{q^{n-\ell}+\eta((-1)^{n})}{q^{n}+\eta((-1)^{n})}A_{2n,2k_0+\ell,\ell,q}^S.$$
If $n$ is even then $\eta((-1)^{n})=1$. If $n$ is odd, there are two possibilities: $\eta((-1)^{n})=1$ if $q\equiv 1\pmod 4$, and $\eta((-1)^{n})=-1$ if $q\equiv 3\pmod 4$.

The full-space case $k=2n$, $\ell=0$ has both counts equal to $1$.
We summarize the result in the following corollary.

\begin{corollary}
If $n$, $k$ and $\ell$ are integers such that $0\le\ell\le k\le 2n-\ell$, $k-\ell$ even, then $A_{2n,k,\ell,q}^E\le A_{2n,k,\ell,q}^S$, as the equality holds if and only if $\ell=0$.
\end{corollary}

\subsection{Euclidean vs Hermitian hulls}

We will use the formula for $A_{n,k,\ell,q^2}^H$ in the following form:

\begin{align*}
A_{n,k,\ell,q^2}^H&=
\left(\prod_{i=1}^\ell M_H(n,k,\ell,i)\right)q^{k_0(n-k_0)}\prod_{i=1}^{k_0}\frac{q^{n-k_0+i}-(-1)^{n-k_0+i}}{q^i-(-1)^i}\\
&=
\left(\prod_{i=1}^\ell M_H(n,k,\ell,i)\right)q^{k_0(n-k_0)}\prod_{i=1}^{k_0}\frac{q^{n-i+1}-(-1)^{n-i+1}}{q^i-(-1)^i},
\end{align*}
where $$M_H(n,k,\ell,i)=\frac{(q^{n-k_0-2i+2}-(-1)^{n-k_0})(q^{n-k_0-2i+1}+(-1)^{n-k_0})}{q^{2k_0}(q^{2i}-1)}.$$

It follows that
$$\frac{M^H_{n,k,\ell,i}}{M^E_{n,k,\ell,q^2,i}}=\left\{\begin{array}{ll}
\frac{q^{n-k_0-2i+2}+1}{q^{n-k_0-2i+1}+1} &\mbox{if} \ n-k_0 \ \mbox{is odd,}\\
\frac{q^{n-k_0-2i+1}+1}{q^{n-k_0-2i+2}+1}& \mbox{if} \ n-k_0 \ \mbox{is even}.
\end{array}\right.$$

Let $2k\le n$ and $0\le\ell\le k-1$. Now
$$\frac{A^H_{n,k,\ell,q^2}}{A^E_{n,k,\ell,q^2}}=\frac{\alpha^H_{n,k,\ell,q^2}(q^{\ell+1}-1)A^H_{n,k,\ell+1,q^2}}{\alpha^E_{n,k,\ell,q^2}(q^{2(\ell+1)}-1)A^E_{n,k,\ell+1,q^2}} =\frac{\alpha^H_{n,k,\ell,q^2}A^H_{n,k,\ell+1,q^2}}{\alpha^E_{n,k,\ell,q^2}(q^{\ell+1}+1)A^E_{n,k,\ell+1,q^2}}.$$

Let us evaluate $$\frac{\alpha^H_{n,k,\ell,q^2}}{\alpha^E_{n,k,\ell,q^2}(q^{\ell+1}+1)}
=\frac{q^{n-2\ell-1}} {\alpha^E_{n,k,\ell,q^2}(q^{n-k-\ell}-(-1)^{n-k-\ell})(q^{k-\ell}-(-1)^{k-\ell})}.$$

\begin{equation}
\alpha_{n,k,\ell,q^2}^E=\left\{\begin{array}{ll}
\frac{q^{2n-2k-2\ell}}{q^{2n-2k-2\ell}-1}&\mbox{if} \ n \ \mbox{is odd}, k-\ell \ \mbox{is odd},\\
\frac{q^{2k-2\ell}}{q^{2k-2\ell}-1}&\mbox{if} \ n \ \mbox{is odd}, k-\ell \ \mbox{is even},\\
\frac{q^{2n-2\ell-2}}{q^{2n-2\ell-2}-1}&\mbox{if} \ n \ \mbox{is even}, k-\ell \ \mbox{is odd}, \ \mbox{and} \ q \ \mbox{is even},\\
\frac{q^{n-2\ell-2}}{q^{n-2\ell-2}+1}&\mbox{if} \ n \ \mbox{is even}, k-\ell \ \mbox{is odd}, \ \mbox{and} \ q \ \mbox{is odd},\\
\frac{q^{2n-2\ell}-1}{q^{2\ell}(q^{2n-2k-2\ell}-1)(q^{2k-2\ell}-1)}&\mbox{if} \ n \ \mbox{is even}, k-\ell \ \mbox{is even}, \ \mbox{and} \ q \ \mbox{is even},\\
\frac{q^{n-2\ell}(q^{n-2\ell}+1)}{(q^{2n-2k-2\ell}-1)(q^{2k-2\ell}-1)}&\mbox{if} \ n \ \mbox{is even}, k-\ell \ \mbox{is even}, \ \mbox{and} \ q \ \mbox{is odd}.
\end{array}\right.
\end{equation}

Note that since the Euclidean codes are taken over $\F_{q^2}$, in which $-1$ is a square, we
have $\eta((-1)^{n/2})=1$.

Denote by $R$ the quantity evaluated above,
\[
R:=\frac{\alpha^H_{n,k,\ell,q^2}}{\alpha^E_{n,k,\ell,q^2}(q^{\ell+1}+1)}
=\frac{q^{n-2\ell-1}}{\alpha^E_{n,k,\ell,q^2}\,(q^{n-k-\ell}-(-1)^{n-k-\ell})(q^{k-\ell}-(-1)^{k-\ell})}.
\]
Substituting each of the six forms of $\alpha^E_{n,k,\ell,q^2}$ and cancelling the common factors gives, with $k_0=k-\ell$:
\begin{enumerate}
\item[(1)] $n$ odd, $k_0$ odd:\quad
$\displaystyle R=\frac{q^{n-k-\ell}+1}{q^{\,n-2k+1}(q^{k-\ell}+1)}<\frac{q^{n-k-\ell}+1}{q^{\,n-k-\ell+1}}<1$.
\item[(2)] $n$ odd, $k_0$ even:\quad
$\displaystyle R=\frac{q^{\,n-2k-1}(q^{k-\ell}+1)}{q^{n-k-\ell}+1}<\frac{q^{k-\ell}+1}{q^{\,k-\ell+1}}<1$.
\item[(3)] $n,q$ even, $k_0$ odd:\quad
$\displaystyle R=\frac{q^{\,2n-2\ell-2}-1}{q^{\,n-1}(q^{n-k-\ell}+1)(q^{k-\ell}+1)}<q^{-1}<1$.
\item[(4)] $n,q,k_0$ even (so $k_0\ge2$):
\[
\begin{aligned}
R&=\frac{q^{\,n-1}(q^{n-k-\ell}+1)(q^{k-\ell}+1)}{q^{\,2n-2\ell}-1}\\
&=\frac{q^{-1}(1+q^{-(n-k-\ell)})(1+q^{-(k-\ell)})}{1-q^{-(2n-2\ell)}}\\
&\le\frac{\tfrac12(1+\tfrac14)^2}{1-\tfrac1{256}}=\frac{40}{51}<1.
\end{aligned}
\]
since $k-\ell\ge2$, $n-k-\ell\ge2$ and $2n-2\ell\ge8$.
\item[(5)] $n$ even, $q$ odd, $k_0$ odd:\quad
$\displaystyle R=\frac{q\,(q^{\,n-2\ell-2}+1)}{(q^{n-k-\ell}+1)(q^{k-\ell}+1)}<\frac{q^{\,n-2\ell-1}+q}{q^{\,n-2\ell}}\le1$.
\item[(6)] $n$ even, $q$ odd, $k_0$ even:\quad
$\displaystyle R=\frac{(q^{n-k-\ell}+1)(q^{k-\ell}+1)}{q\,(q^{\,n-2\ell}+1)}\le\frac{q^{\,n-2\ell}+q^{n-k-\ell}+q^{k-\ell}+1}{q^{\,n-2\ell+1}}<1$.
\end{enumerate}

In all cases the considered ratio is less than 1, hence $$\frac{A^H_{n,k,\ell,q^2}}{A^E_{n,k,\ell,q^2}}<\frac{A^H_{n,k,\ell+1,q^2}}{A^E_{n,k,\ell+1,q^2}}.$$

This means that 
$$\frac{A^H_{n,k,0,q^2}}{A^E_{n,k,0,q^2}}<\frac{A^H_{n,k,1,q^2}}{A^E_{n,k,1,q^2}}<\cdots<\frac{A^H_{n,k,k,q^2}}{A^E_{n,k,k,q^2}}.$$

Therefore, we will calculate the smallest values of the considered ratio.

The following elementary properties of the function
$f(a)=\dfrac{q^{a+1}+1}{q^{a}+1}$ will be used in the proof of the next theorem.

\begin{lemma}\label{lem:f}
Let $q\ge 2$ be a prime power and let $a$ be a non-negative integer. Put
$f(a)=\dfrac{q^{a+1}+1}{q^{a}+1}$. Then
\begin{itemize}
\item[\textup{(f1)}] $f(a)=q-\dfrac{q-1}{q^{a}+1}$; in particular $1<f(a)<q$, and
$f$ is strictly increasing in $a$;
\item[\textup{(f2)}] for all integers $1\le s\le t$,
$\displaystyle\prod_{a=s}^{t}f(a)=\frac{q^{t+1}+1}{q^{s}+1}$;
\item[\textup{(f3)}] for $P_A(s)=\prod_{i=1}^{s} f(n-2i+1)\,f(2i-1)$ and $P_B(s)=\prod_{i=1}^{s} f(n-2i)\,f(2i-1)$,  if $s\ge 1$ and $n\ge 4s$, the following 
telescoping bounds hold
\begin{equation}\label{eq:PAB}
\frac{q^n+1}{q^{n-2s}+1}>P_A(s)>\frac{q^{\,2s+1}+1}{q+1},
\qquad
\frac{q^{n-1}+1}{q^{n-2s-1}+1}\ge P_B(s)\ge \frac{q^{\,2s+1}+1}{q+1}.
\end{equation}
\end{itemize}
\end{lemma}

\begin{proof}
(f1) A direct computation gives
$q-\dfrac{q-1}{q^{a}+1}=\dfrac{q(q^{a}+1)-(q-1)}{q^{a}+1}=\dfrac{q^{a+1}+1}{q^{a}+1}=f(a)$.
Since $q\ge2$ we have $q^{a}<q^{a+1}$, so
$q^{a}+1<q^{a+1}+1<q^{a+1}+q=q(q^{a}+1)$, whence $1<f(a)<q$. Since $q-1>0$ and
$q^{a}+1$ increases strictly with $a$, the term $\dfrac{q-1}{q^{a}+1}$ decreases
strictly, so $f$ increases strictly.

(f2) The product telescopes, because the numerator $q^{a+1}+1$ of the factor
indexed by $a$ coincides with the denominator of the factor indexed by $a+1$:
\[
\prod_{a=s}^{t}\frac{q^{a+1}+1}{q^{a}+1}
=\frac{q^{s+1}+1}{q^{s}+1}\cdot\frac{q^{s+2}+1}{q^{s+1}+1}\cdots
\frac{q^{t+1}+1}{q^{t}+1}
=\frac{q^{t+1}+1}{q^{s}+1}.
\]

(f3) For $1\le i\le s$ we have $n-2i+1>2i$,
so $f(n-2i+1)>f(2i)$ by (f1), and hence, by (f2),
$$P_A(s)=\prod_{i=1}^{s} f(2i-1)f(n-2i+1)>\prod_{i=1}^{s} f(2i-1)f(2i)=\prod_{a=1}^{2s} f(a)=\frac{q^{\,2s+1}+1}{q+1}.$$
Similarly, $f(2i-1)<f(n-2i)$ and therefore
$$P_A(s)<\prod_{i=1}^{s} f(n-2i)f(n-2i+1)=\prod_{a=n-2s}^{n-1} f(a)=\frac{q^{\,n}+1}{q^{n-2s}+1}.$$
In $P_B$, for $1\le i\le s$ we have $n-2i\ge 2i$ (tightest at $i=s$), so $f(n-2i)\ge f(2i)$ and, again by (f2),
$P_B(s)\ge\prod_{a=1}^{2s}f(a)=\frac{q^{\,2s+1}+1}{q+1}$. For the last inequality we use that $f(2i-1)\le f(n-2i-1)$ and therefore $P_B(s)\le \frac{q^{n-1}+1}{q^{n-2s-1}+1}$.
\end{proof}

For the following calculations, we need to evaluate the following ratio
$$R_{n,k_0}=\frac{\prod_{i=1}^{k_0}\frac{q^{n-i+1}-(-1)^{n-i+1}}{q^i-(-1)^i}}{{n_1\brack k_1}_{q^4}}, \ n=2n_1+e_1, \ e_1\in\{1,2\}, \ k_0=2k_1+e, \ e\in\{0,1\}.$$

\begin{align*}
  \prod_{i=1}^{k_0}\frac{q^{n-i+1}-(-1)^{n-i+1}}{q^i-(-1)^i} & =\prod_{i=1}^{2k_1+e}\frac{q^{n-i+1}-(-1)^{n-i+1}}{q^i-(-1)^i}\\
   & =P_e\prod_{i=1}^{k_1}\frac{(q^{n-2i+1}+(-1)^{n})(q^{n-2i+2}-(-1)^{n})}{(q^{2i}-1)(q^{2i-1}+1)}.
\end{align*}
$${n_1\brack k_1}_{q^4}=\prod_{i=1}^{k_1}\frac{(q^{2n_1-2i+2}+1)(q^{2n_1-2i+2}-1)}{(q^{2i}+1)(q^{2i}-1)}.$$

If $n=2n_1+1$ is odd then
$$R_{n,k_0}=P_e\prod_{i=1}^{k_1}\frac{(q^{n-2i+2}+1)(q^{2i}+1)}{(q^{n-2i+1}+1)(q^{2i-1}+1)}=P_e P_A(k_1)$$

If $n=2n_1+2$ then
\begin{align*}
R_{n,k_0}&=P_e\prod_{i=1}^{k_1}\frac{(q^{n-2i+1}+1)(q^{n-2i+2}-1)(q^{2i}+1)}{(q^{n-2i}+1)(q^{n-2i}-1)(q^{2i-1}+1)}\\
&=\frac{P_e(q^n-1)}{q^{n-2k_1}-1}\prod_{i=1}^{k_1}\frac{(q^{n-2i+1}+1)(q^{2i}+1)}{(q^{n-2i}+1)(q^{2i-1}+1)}
=\frac{P_e(q^n-1)}{q^{n-2k_1}-1}P_B(k_1).
\end{align*}

Here $P_0=1$ and $P_1=\frac{q^{n-k_0+1}-(-1)^{n}}{q^{k_0}+1}$.

\begin{theorem}\label{thm:HvsE} If $k\le n/2$ are positive integers then
$$\frac{A_{n,k,0,q^2}^H}{A_{n,k,0,q^2}^E}<1<\frac{A_{n,k,1,q^2}^H}{A_{n,k,1,q^2}^E}.$$
\end{theorem}

\begin{proof} We use the following formulae
$$A_{n,k,0,q^2}^H=q^{k(n-k)}\prod_{i=1}^k\frac{q^{n-i+1}-(-1)^{n-i+1}}{q^i-(-1)^i},$$

\begin{equation}
A_{n,k,0,q^2}^E=\left\{\begin{array}{ll}
q^{k(n-k)+n-k}{(n-1)/2\brack(k-1)/2 }_{q^4}&\mbox{if} \ n \ \mbox{is odd}, k \ \mbox{is odd},\\
q^{k(n-k)+k}{(n-1)/2\brack k/2 }_{q^4}&\mbox{if} \ n \ \mbox{is odd}, k \ \mbox{is even},\\
q^{k(n-k)+n-1}{n/2-1\brack(k-1)/2 }_{q^4}&\mbox{if} \ n \ \mbox{is even}, k \ \mbox{is odd}, \ q \ \mbox{is even},\\
q^{k(n-k)-1}(q^{n}-1){n/2-1\brack(k-1)/2 }_{q^4}&\mbox{if} \ n \ \mbox{is even}, k \ \mbox{is odd}, \ q \ \mbox{is odd},\\
q^{k(n-k)}\frac{q^{2n}-1}{q^{2n-2k}-1}{n/2-1\brack k/2 }_{q^4}&\mbox{if} \ n \ \mbox{is even}, k \ \mbox{is even}, \ q \ \mbox{is even},\\
q^{k(n-k)}\frac{(q^{n}+1)(q^{2n}-1)}{(q^{n}+1)(q^{2n-2k}-1)}{n/2-1\brack k/2 }_{q^4}&\mbox{if} \ n \ \mbox{is even}, k \ \mbox{is even}, \ q \ \mbox{is odd}.\\
\end{array}\right.
\end{equation}

Hence
\begin{equation}
\frac{A_{n,k,0,q^2}^H}{A_{n,k,0,q^2}^E}=\left\{\begin{array}{ll}
\frac{1}{q^{n-k}} R_{n,k} = \frac{q^{n-k+1}+1}{q^{n-k}(q^{k}+1)} P_A(k_1)&\mbox{if} \ n \ \mbox{is odd}, k \ \mbox{is odd},\\
\frac{1}{q^{k}}R_{n,k} =\frac{1}{q^{k}}P_A(k_1)&\mbox{if} \ n \ \mbox{is odd}, k \ \mbox{is even},\\
\frac{1}{q^{n-1}}R_{n,k} =\frac{q^n-1}{q^{n-1}(q^{k}+1)}P_B(k_1) &\mbox{if} \ n \ \mbox{is even}, k \ \mbox{is odd}, \ q \ \mbox{is even},\\
\frac{1}{q^{-1}(q^{n}-1)}R_{n,k} = \frac{1}{q^{-1}(q^{k}+1)}P_B(k_1) &\mbox{if} \ n \ \mbox{is even}, k \ \mbox{is odd}, \ q \ \mbox{is odd},\\
\frac{q^{2n-2k}-1}{q^{2n}-1}R_{n,k} =\frac{q^{n-k}+1}{q^{n}+1}P_B(k_1) &\mbox{if} \ n \ \mbox{is even}, k \ \mbox{is even}, \ q \ \mbox{is even},\\
\frac{(q^{n}+1)(q^{2n-2k}-1)}{(q^{n}+1)(q^{2n}-1)}R_{n,k} =\frac{q^{n-k}+1}{q^{n}+1}P_B(k_1) &\mbox{if} \ n \ \mbox{is even}, k \ \mbox{is even}, \ q \ \mbox{is odd}.\\
\end{array}\right.
\end{equation}

Using the inequalities from Lemma \ref{lem:f}, we obtain
$$\frac{q^{n-k+1}+1}{q^{n-k}(q^{k}+1)} P_A(k_1)\le\frac{q^n+1}{q^{n-k}(q^{k}+1)}<1,$$
if $k=2k_1+1$, and
$$\frac{1}{q^{k}}P_A(k_1)<\frac{q^n+1}{q^k(q^{n-k}+1)}<1,$$
if $k=2k_1$.

For even $n$ we have
\begin{eqnarray*}
\frac{q^{n-k}+1}{q^{n}+1}P_B(k_1)&\le &\frac{(q^{n-k}+1)(q^{n-1}+1)}{(q^{n}+1)(q^{n-k-1}+1)}=\frac{q^{2n-k-1}+1+q^{n-k-1}(q^{k}+q)}{q^{2n-k-1}+1+q^{n-k-1}(q^{k+1}+1)}<1,\\
\frac{1}{q^{-1}(q^{k}+1)}P_B(k_1)&\le &\frac{q^{n}+q}{(q^{k}+1)(q^{n-k}+1)}<1,\\
\frac{q^n-1}{q^{n-1}(q^{k}+1)}P_B(k_1)&\le &\frac{(q^n-1)(q^{n-1}+1)}{q^{n-1}(q^{k}+1)(q^{n-k}+1)}<\frac{q(q^{n-1}+1)}{(q^{k}+1)(q^{n-k}+1)}<1.
\end{eqnarray*}

This proves the first inequality.

For the second inequality we use the following formula
$$\Phi=\frac{A^H_{n,k,1,q^2}}{A^E_{n,k,1,q^2}}=\frac{\alpha^E_{n,k,0,q^2}(q+1)A^H_{n,k,0,q^2}}{\alpha^H_{n,k,0,q^2}A^E_{n,k,0,q^2}}.$$

\begin{equation}
\alpha_{n,k,0,q^2}^E=\left\{\begin{array}{ll}
\frac{q^{2n-2k}}{q^{2n-2k}-1}&\mbox{if} \ n \ \mbox{is odd}, k \ \mbox{is odd},\\
\frac{q^{2k}}{q^{2k}-1}&\mbox{if} \ n \ \mbox{is odd}, k \ \mbox{is even},\\
\frac{q^{2n-2}}{q^{2n-2}-1}&\mbox{if} \ n \ \mbox{is even}, k \ \mbox{is odd}, \ \mbox{and} \ q \ \mbox{is even},\\
\frac{q^{n-2}}{q^{n-2}+1}&\mbox{if} \ n \ \mbox{is even}, k \ \mbox{is odd}, \ \mbox{and} \ q \ \mbox{is odd},\\
\frac{q^{2n}-1}{(q^{2n-2k}-1)(q^{2k}-1)}&\mbox{if} \ n \ \mbox{is even}, k \ \mbox{is even}, \ \mbox{and} \ q \ \mbox{is even},\\
\frac{q^{n}(q^{n}+1)}{(q^{2n-2k}-1)(q^{2k}-1)}&\mbox{if} \ n \ \mbox{is even}, k \ \mbox{is even}, \ \mbox{and} \ q \ \mbox{is odd}.
\end{array}\right.
\end{equation}

Let us evaluate 
$$\frac{\alpha^E_{n,k,0,q^2}(q+1)}{\alpha^H_{n,k,0,q^2}}
=\frac{\alpha^E_{n,k,0,q^2}(q^{n-k}-(-1)^{n-k})(q^{k}-(-1)^{k})}{q^{n-1}} .$$
\[
=\left\{\begin{array}{ll}
\frac{q^{2n-2k}(q^{n-k}-1)(q^{k}+1)}{q^{n-1}(q^{2n-2k}-1)}=\frac{q^{n-2k+1}(q^{k}+1)}{(q^{n-k}+1)}&\mbox{if} \ n \ \mbox{is odd}, k \ \mbox{is odd},\\
\frac{q^{2k}(q^{n-k}+1)(q^{k}-1)}{q^{n-1}(q^{2k}-1)} =\frac{(q^{n-k}+1)}{q^{n-2k-1}(q^{k}+1)}&\mbox{if} \ n \ \mbox{is odd}, k \ \mbox{is even},\\
\frac{q^{2n-2}(q^{n-k}+1)(q^{k}+1)}{q^{n-1}(q^{2n-2}-1)} =\frac{q^{n-1}(q^{n-k}+1)(q^{k}+1)}{(q^{2n-2}-1)}&\mbox{if} \ n \ \mbox{is even}, k \ \mbox{is odd}, \ \mbox{and} \ q \ \mbox{is even},\\
\frac{q^{n-2}(q^{n-k}+1)(q^{k}+1)}{q^{n-1}(q^{n-2}+1)}=\frac{(q^{n-k}+1)(q^{k}+1)}{q(q^{n-2}+1)} &\mbox{if} \ n \ \mbox{is even}, k \ \mbox{is odd}, \ \mbox{and} \ q \ \mbox{is odd},\\
\frac{(q^{2n}-1)(q^{n-k}-1)(q^{k}-1)}{q^{n-1}(q^{2n-2k}-1)(q^{2k}-1)} =\frac{(q^{2n}-1)}{q^{n-1}(q^{n-k}+1)(q^{k}+1)}&\mbox{if} \ n \ \mbox{is even}, k \ \mbox{is even}, \ \mbox{and} \ q \ \mbox{is even},\\
\frac{q^{n}(q^{n}+1)(q^{n-k}-1)(q^{k}-1)}{q^{n-1}(q^{2n-2k}-1)(q^{2k}-1)} =\frac{q(q^{n}+1)}{(q^{n-k}+1)(q^{k}+1)}&\mbox{if} \ n \ \mbox{is even}, k \ \mbox{is even}, \ \mbox{and} \ q \ \mbox{is odd}.
\end{array}\right.
\]

It follows that
$$\frac{A^H_{n,k,1,q^2}}{A^E_{n,k,1,q^2}}=
\left\{\begin{array}{ll}
\frac{q^{n-2k+1}(q^{k}+1)}{(q^{n-k}+1)}\frac{q^{n-k+1}+1}{q^{n-k}(q^{k}+1)} P_A(k_1)&\mbox{if} \ n \ \mbox{is odd}, k \ \mbox{is odd},\\
\frac{(q^{n-k}+1)}{q^{n-2k-1}(q^{k}+1)}\frac{1}{q^{k}}P_A(k_1)&\mbox{if} \ n \ \mbox{is odd}, k \ \mbox{is even},\\
\frac{q^{n-1}(q^{n-k}+1)(q^{k}+1)}{(q^{2n-2}-1)}\frac{q^n-1}{q^{n-1}(q^{k}+1)}P_B(k_1)&\mbox{if} \ n \ \mbox{is even}, k \ \mbox{is odd}, \ \mbox{and} \ q \ \mbox{is even},\\
\frac{(q^{n-k}+1)(q^{k}+1)}{q(q^{n-2}+1)}\frac{1}{q^{-1}(q^{k}+1)}P_B(k_1) &\mbox{if} \ n \ \mbox{is even}, k \ \mbox{is odd}, \ \mbox{and} \ q \ \mbox{is odd},\\
\frac{(q^{2n}-1)}{q^{n-1}(q^{n-k}+1)(q^{k}+1)}\frac{q^{n-k}+1}{q^{n}+1}P_B(k_1)&\mbox{if} \ n \ \mbox{is even}, k \ \mbox{is even}, \ \mbox{and} \ q \ \mbox{is even},\\
\frac{q(q^{n}+1)}{(q^{n-k}+1)(q^{k}+1)}\frac{q^{n-k}+1}{q^{n}+1}P_B(k_1)&\mbox{if} \ n \ \mbox{is even}, k \ \mbox{is even}, \ \mbox{and} \ q \ \mbox{is odd}.
\end{array}\right.$$

$$\frac{A^H_{n,k,1,q^2}}{A^E_{n,k,1,q^2}}=
\left\{\begin{array}{ll}
\frac{q^{n-k+1}+1}{q^{k-1}(q^{n-k}+1)} P_A(k_1)&\mbox{if} \ n \ \mbox{is odd}, k \ \mbox{is odd},\\
\frac{(q^{n-k}+1)}{q^{n-k-1}(q^{k}+1)}P_A(k_1)&\mbox{if} \ n \ \mbox{is odd}, k \ \mbox{is even},\\
\frac{(q^{n-k}+1)(q^n-1)}{(q^{2n-2}-1)}P_B(k_1)&\mbox{if} \ n \ \mbox{is even}, k \ \mbox{is odd}, \ \mbox{and} \ q \ \mbox{is even},\\
\frac{(q^{n-k}+1)}{(q^{n-2}+1)}P_B(k_1) &\mbox{if} \ n \ \mbox{is even}, k \ \mbox{is odd}, \ \mbox{and} \ q \ \mbox{is odd},\\
\frac{(q^{n}-1)}{q^{n-1}(q^{k}+1)}P_B(k_1)&\mbox{if} \ n \ \mbox{is even}, k \ \mbox{is even}, \ \mbox{and} \ q \ \mbox{is even},\\
\frac{q}{(q^{k}+1)}P_B(k_1)&\mbox{if} \ n \ \mbox{is even}, k \ \mbox{is even}, \ \mbox{and} \ q \ \mbox{is odd}.
\end{array}\right.$$

As in the previous proof, we consider the six cases separately. Using the lower bounds for $P_A$ and $P_B$, we obtain
$\frac{A^H_{n,k,1,q^2}}{A^E_{n,k,1,q^2}}\ge\frac{A}{B}$,
where $A$ and $B$ are integer-valued functions depending on $q$, $n$, and $k$. Since $A$ and $B$ are positive, we have
$\frac{A}{B}>1$ if and only if $A>B$.

\begin{enumerate}
\item If $n$ and $k$ are odd then
$$\frac{q^{n-k+1}+1}{q^{k-1}(q^{n-k}+1)} P_A(k_1)\ge\frac{q^{n-k+1}+1}{q^{k-1}(q^{n-k}+1)}\frac{q^{k}+1}{q+1}.$$
The last quantity is $\ge1$ if and only if
$(q^{n-k+1}+1)(q^{k}+1)\ge q^{k-1}(q+1)(q^{n-k}+1)$. Expanding both sides,
$$(q^{n-k+1}+1)(q^{k}+1)=q^{n+1}+q^{n-k+1}+q^{k}+1,$$
$$q^{k-1}(q+1)(q^{n-k}+1)=q^{k-1}(q^{n-k+1}+q^{n-k}+q+1)=q^{n}+q^{n-1}+q^{k}+q^{k-1},$$
so their difference is
$$q^{n-1}(q^{2}-q-1)+\bigl(q^{n-k+1}-q^{k-1}\bigr)+1>0,$$
because $q^{2}-q-1\ge1$ and $n-k+1\ge k-1$ (as $n\ge2k$).

\item If $n$ is odd and $k$ is even, then $2k_1+1=k+1$, and
$$\Phi=\frac{q^{n-k}+1}{q^{n-k-1}(q^{k}+1)}\,P_A
>\frac{q^{n-k}+1}{q^{n-k-1}(q^{k}+1)}\cdot\frac{q^{k+1}+1}{q+1}.$$
The last quantity is $\ge1$ if and only if
$(q^{n-k}+1)(q^{k+1}+1)\ge q^{n-k-1}(q+1)(q^{k}+1)$. Since
$$(q^{n-k}+1)(q^{k+1}+1)=q^{n+1}+q^{n-k}+q^{k+1}+1,$$
$$q^{n-k-1}(q+1)(q^{k}+1)=q^{n-k-1}(q^{k+1}+q^{k}+q+1)=q^{n}+q^{n-1}+q^{n-k}+q^{n-k-1},$$
the difference equals
$$q^{n-1}(q^{2}-q-1)+q^{k+1}-q^{n-k-1}+1.$$
As $k\ge2$ we have $q^{n-k-1}\le q^{n-3}$, and $q^{n-1}(q^{2}-q-1)\ge q^{n-1}\ge q^{n-3}\ge q^{n-k-1}$, so the difference is positive.

\item If $n$ and $q$ are even, but $k=2k_1+1$ is odd, then
\[\begin{aligned}
\frac{(q^{n-k}+1)(q^n-1)}{q^{2n-2}-1}P_B(k_1)
&\ge \frac{(q^{n-k}+1)(q^n-1)(q^{k}+1)}{(q^{2n-2}-1)(q+1)}\\
&=\frac{q^n-1}{q^{n-1}-1}\frac{(q^{n-k}+1)(q^{k}+1)}{(q^{n-1}+1)(q+1)}.
\end{aligned}\]
This is $\ge1$ if and only if $(q^{n-k}+1)(q^{k}+1)(q^{n}-1)\ge(q^{2n-2}-1)(q+1)$. Now
$$(q^{n-k}+1)(q^{k}+1)=q^{n}+q^{n-k}+q^{k}+1,$$
$$(q^{n}+q^{n-k}+q^{k}+1)(q^{n}-1)=q^{2n}+q^{2n-k}+q^{n+k}-q^{n-k}-q^{k}-1,$$
$$(q^{2n-2}-1)(q+1)=q^{2n-1}+q^{2n-2}-q-1,$$
so the difference of the two sides is
$$q^{2n-2}(q^{2}-q-1)+\bigl(q^{2n-k}-q^{k}\bigr)+\bigl(q^{n+k}-q^{n-k}\bigr)+q>0,$$
since $2n-k\ge k$ and $n+k\ge n-k$.

\item Let $n=2n_1$, $q$ be odd and $k$ be odd. Then $2k_1+1=k$, and
$$\Phi=\frac{q^{n-k}+1}{q^{n-2}+1}\,P_B
\ge\frac{q^{n-k}+1}{q^{n-2}+1}\cdot\frac{q^{k}+1}{q+1}.$$
This is $\ge1$ if and only if $(q^{n-k}+1)(q^{k}+1)\ge(q^{n-2}+1)(q+1)$. From
$$(q^{n-k}+1)(q^{k}+1)=q^{n}+q^{n-k}+q^{k}+1,\qquad
(q^{n-2}+1)(q+1)=q^{n-1}+q^{n-2}+q+1,$$
the difference is
$$q^{n-2}(q^{2}-q-1)+\bigl(q^{n-k}-q\bigr)+q^{k}>0,$$
since $n-k\ge k\ge1$ gives $q^{n-k}\ge q$.

\item Let $n$, $q$ and $k$ be even. Then $2k_1+1=k+1$, and
$$\Phi=\frac{q^{n}-1}{q^{n-1}(q^{k}+1)}\,P_B(k_1)
\ge\frac{q^{n}-1}{q^{n-1}(q^{k}+1)}\cdot\frac{q^{k+1}+1}{q+1}.$$
This is $\ge1$ if and only if $(q^{n}-1)(q^{k+1}+1)\ge q^{n-1}(q+1)(q^{k}+1)$. Since
$$(q^{n}-1)(q^{k+1}+1)=q^{n+k+1}+q^{n}-q^{k+1}-1,$$
$$q^{n-1}(q+1)(q^{k}+1)=q^{n-1}(q^{k+1}+q^{k}+q+1)=q^{n+k}+q^{n+k-1}+q^{n}+q^{n-1},$$
the difference equals
$$q^{n+k-1}(q^{2}-q-1)-\bigl(q^{n-1}+q^{k+1}+1\bigr).$$
Because $n\ge2k$ and $k\ge2$ we have $k+1\le n-1$, hence
$q^{n-1}+q^{k+1}+1\le 3q^{n-1}$, whereas
$q^{n+k-1}(q^{2}-q-1)\ge q^{n+k-1}\ge q^{n+1}=q^{2}\,q^{n-1}\ge 4q^{n-1}$, so the
difference is positive.

\item For odd $q$ and $k=2k_1$
$$\frac{q}{(q^{k}+1)}P_B(k_1)\ge \frac{q(q^{k+1}+1)}{(q^{k}+1)(q+1)}=\frac{q^{k+2}+q}{q^{k+1}+q^{k}+q+1}>1.$$
Here we use that the difference between the nominator and denominator is $q^{k}(q^{2}-q-1)-1>0,$
which holds for $k\ge2$ and $q\ge2$ (then $q^{k}(q^{2}-q-1)\ge q^{2}\ge4$).
\end{enumerate}

In all cases $\Phi=A^H_{n,k,1,q^2}/A^E_{n,k,1,q^2}>1$, which completes the proof.
\end{proof}

The above theorem shows that for fixed $q$, $n$ and $k\le n/2$, the number of Hermitian LCD $[n,k]_{q^2}$ codes is less than the number of Euclidean LCD codes of the same length and dimension. 

\subsection{Symplectic entanglement-assisted stabilizer codes}

Relaxing the symplectic self-orthogonality required by KKKS-style stabilizer
constructions~\cite{Ketkar} to allow shared entanglement yields the
entanglement-assisted stabilizer formalism of Brun, Devetak, and
Hsieh~\cite{BDH}. The minimum number of pre-shared maximally entangled pairs
needed by such a construction was determined by Wilde and
Brun~\cite[Theorem~1]{WildeBrun}.

\begin{theorem}[{\cite[Theorem~1]{WildeBrun}}]\label{thm:WB}
Let $H = [H_Z \mid H_X]$ be an $(n - k) \times 2n$ binary quantum check matrix of full row rank.
The resulting entanglement-assisted stabilizer code has parameters
$[[n, k + c; c]]_2$, where the optimal number of ebits is
\[
   c = \tfrac{1}{2}\,\rank\bigl(H_X H_Z^T + H_Z H_X^T\bigr).
\]
\end{theorem}

The matrix $H_X H_Z^T + H_Z H_X^T$ is the Gram matrix of the rows of $H$ under
the symplectic form $\langle\cdot,\cdot\rangle_S$ of
Section~\ref{sect:preliminaries}; viewing the row span of $H$ as an
$\F_2$-linear $[2n, m]_2$ code $C = \langle H\rangle$ of dimension $m = n - k$,
the Gram rank equals $m - \dim(\mathcal{H}_S(C))$ and hence
$c = \tfrac{1}{2}\bigl(m - \dim(\mathcal{H}_S(C))\bigr)$. The same formula
holds for $\F_q$-linear codes over an arbitrary finite field $\F_q$ \cite[Theorem~2]{Galindo2019}. Re-expressing in
our parameters: an $[2n, k]_q$ classical code $C$ with symplectic hull
dimension $\ell$ and $k+\ell<2n$ yields an
\[
   [[\,n,\; n - \tfrac{k+\ell}{2},\; d;\; \tfrac{k - \ell}{2}\,]]_q
\]
entanglement-assisted stabilizer code requiring $c = (k - \ell)/2$ pre-shared
maximally entangled pairs, where
\[
   d=\min\{\mathrm{swt}(v): v\in C^{\perp_S}\setminus\mathcal{H}_S(C)\}
\]
and $\mathrm{swt}(a\mid b)=|\{i: (a_i,b_i)\neq (0,0)\}|$ denotes the symplectic weight of $(a\mid b)\in\F_q^{2n}$ \cite[Theorem~2]{Galindo2019}. The condition $k+\ell<2n$ ensures that the code encodes at least one qudit. The boundary cases $\ell = k$ (symplectic self-orthogonal) and
$\ell = 0$ (symplectic LCD) recover the $[[n, n-k, d]]_q$ stabilizer code
of~\cite[Theorem~13]{Ketkar} and the maximum-entanglement
$[[n, n - k/2, d; k/2]]_q$ EAQECC, respectively.

Theorem~\ref{main_symplectic} therefore controls the count of classical codes giving EAQECCs of each
fixed entanglement level:

\begin{corollary}\label{cor:eaqecc-symp-monotone}
Fix $n$, $k$, $q$, and let $c = (k - \ell)/2$ denote the number of pre-shared
maximally entangled pairs in the EAQECC associated above to an $[2n,k]_q$ code
with symplectic hull dimension $\ell$. For $\ell + 2 \le k \le 2n - \ell - 2$
with $k - \ell$ even, the count $A^S_{2n,k,\ell,q}$ of such classical codes
satisfies the ratio decomposition of Theorem~\ref{main_symplectic}.
Consequently, outside the exception family ($q=2$, $\ell=0$,
$4 \le k \le 2n-4$, $k$ even), $A^S_{2n,k,\ell,q}$ is strictly decreasing in
$\ell$, equivalently strictly increasing in $c$: EAQECCs demanding more
pre-shared entanglement arise from a larger pool of classical codes. On that
family the dependence reverses: the maximum-entanglement family ($c = k/2$) is
locally outnumbered by the $c = (k-2)/2$ family.
\end{corollary}

These are counts of classical input subspaces without a distance constraint.
The distance formula above applies when $n-(k+\ell)/2>0$; boundary strata with zero encoded dimension are included only in the classical counts.

\subsection{Entanglement-assisted quantum error-correcting codes}

The Hermitian-hull analogue is given by Guenda, Jitman, and
Gulliver~\cite{Guenda}; the minimum distances of the resulting codes are determined exactly in \cite[Theorem~3]{Galindo2019}.

\begin{theorem}[{\cite[Corollary~3.2]{Guenda}}]\label{thm:GJG}
Let $C$ be a classical $[n,k,d]_{q^2}$ code and let $C^{\perp_H}$ be its
Hermitian dual with parameters $[n, n-k, d^{\perp_H}]_{q^2}$. If $\dim(\mathcal{H}_H(C))<k$, then there exists an
\[
   [[n,\; k - \dim(\mathcal{H}_H(C)),\; d_1;\; n-k-\dim(\mathcal{H}_H(C))]]_q
\]
EAQECC, and if $\dim(\mathcal{H}_H(C))<n-k$, then there exists an
\[
   [[n,\; n-k - \dim(\mathcal{H}_H(C)),\; d_2;\; k-\dim(\mathcal{H}_H(C))]]_q
\]
EAQECC, where
\[
   d_1=\min\{\wt(v): v\in C\setminus\mathcal{H}_H(C)\}\ge d \quad \mbox{and} \quad
   d_2=\min\{\wt(v): v\in C^{\perp_H}\setminus\mathcal{H}_H(C)\}\ge d^{\perp_H}.
\]
\end{theorem}

Substituting $\dim(\mathcal{H}_H(C)) = \ell$, every $[n,k]_{q^2}$ code with
Hermitian hull dimension $\ell<k$ thus yields an
$[[n,\, k - \ell,\, d_1;\, n-k-\ell]]_q$ EAQECC. In the limiting case $\ell = 0$
(Hermitian LCD) we have $d_1=d$, and the code is a maximal-entanglement $[[n,k,d;n-k]]_q$ EAQECC. In the other limiting case $\ell = k$ (Hermitian self-orthogonal) the first construction does not apply, whereas the second one gives the usual $[[n,n-2k,d_2]]_q$ stabilizer code obtained from a Hermitian self-orthogonal code, without entanglement assistance (if $k<n/2$). The ratio decomposition of Theorem~\ref{main_hermitian}
thus controls the count of classical codes giving EAQECCs with exactly $n - k - \ell$ pre-shared
maximally entangled pairs:

\begin{corollary}\label{cor:eaqecc-monotone}
Fix $n$, $k$, $q$, and let $c = n - k - \ell$ denote the number of maximally
entangled pairs required by the first EAQECC of Theorem~\ref{thm:GJG} corresponding
to an $[n,k]_{q^2}$ code with Hermitian hull dimension $\ell$. For
$\ell + 1 \le k \le n - \ell - 1$, the count $A^H_{n,k,\ell,q^2}$ of such
classical codes satisfies the ratio decomposition of
Theorem~\ref{main_hermitian}, and outside the boundary family
$\{(q, k) = (2, 1), (2, n-1) : n \text{ even}\}$ at $\ell = 0$,
$A^H_{n,k,\ell,q^2}$ is strictly decreasing in $\ell$, equivalently strictly
\emph{increasing} in the entanglement parameter $c$.
\end{corollary}

Corollary~\ref{cor:eaqecc-monotone} quantifies how rapidly the classical input pool
shrinks as one demands fewer pre-shared maximally entangled pairs (and thereby
a larger Hermitian hull), with the lower bound $\alpha^H \ge q/(q+1)$ from
Theorem~\ref{main_hermitian} improving on the Euclidean analogue.
These counts impose no distance constraint and are not counts of inequivalent EAQECCs.
The encoded dimension $k-\ell$ also varies with $\ell$; the first construction and its stated distance formula apply only when $k>\ell$.

\section{Computational Results}

We classified linear codes over $\mathbb F_4$ of lengths $6\leq n\leq 11$ according to the dimension of their Hermitian hull. The numbers of inequivalent $[n,k]_4$ linear codes obtained for each possible hull dimension are presented in Table~\ref{tab:hermitian-hull}. In our computations, we use general equivalence of linear codes, whereas recent classification results for LCD codes typically consider monomial equivalence. This choice is consistent with the equivalence relation used in our previous computational classifications and allows us to treat the codes in a uniform way.

We recall the definition of general equivalence.

\begin{definition}
Two linear $q$-ary codes $\mathcal{C}_1$ and $\mathcal{C}_2$ are equivalent if the codewords of $\mathcal{C}_2$ can be obtained from the codewords of $\mathcal{C}_1$ by a sequence of transformations of the following types:
\begin{itemize}
\item[(1)] permutation of the coordinate positions;
\item[(2)] multiplication of the elements in a given coordinate position by a nonzero element of $\mathbb F_q$;
\item[(3)] application of a field automorphism to the elements in all coordinate positions.
\end{itemize}
\end{definition}

An automorphism of a linear code $\mathcal C$ is a sequence of transformations of types (1)--(3) that maps each codeword of $\mathcal C$ onto a codeword of the same code. 
In matrix form, two codes $\mathcal{C}_1$ and $\mathcal{C}_2$ of the same length over $\F_q$ are equivalent if there exist a monomial matrix $M$ and a field automorphism $\gamma$ such that
$\mathcal{C}_2=\mathcal{C}_1M\gamma$ \cite{HP}. The set of all such automorphisms forms a group, called the automorphism group of the code.

The use of general equivalence is particularly relevant in the present setting. For a field $\F_{q^2}$, it is known that, for $q>2$, every linear code is monomially equivalent to an Hermitian LCD code \cite{Carlet_Pellikaan}, whereas the corresponding statement does not hold for $q=2$. Thus, in the quaternary case, the choice of the equivalence relation is significant. The codes were generated using the \texttt{Generation} module, and the equivalence classes were determined using the \texttt{LCequivalence} module of the \texttt{QextNewEdition} software package \cite{Generation}.

As a check on our computations, we compared our results with the classification of quaternary Hermitian LCD codes by Araya and Harada \cite{Araya-Harada-F4}. Their classification is given up to monomial equivalence. Since the authors also provide a database containing generator matrices of the classified codes, we examined these codes using our equivalence relation. The numbers of Hermitian LCD codes obtained in this way agree with the corresponding numbers in their classification. This provides an independent verification of the entries corresponding to Hermitian LCD codes in Table~\ref{tab:hermitian-hull}.

We performed a further verification by comparing the numbers of Hermitian self-orthogonal codes with previously published computational results \cite{BO-SOcodes}. The agreement with these results provides an additional check on our enumeration and on the implementation of the equivalence testing.

Finally, we note that Hermitian orthogonality is preserved under general equivalence over $\F_4$. Indeed, $a\bar a=a^3=1$ for every $a\in\F_4^*$, so coordinate permutations and nonzero coordinate scalings preserve the Hermitian inner product. The nontrivial automorphism of $\F_4$ is the Frobenius automorphism $x\mapsto x^2$, and the Hermitian inner product
is preserved up to application of this automorphism. Consequently, if two codes are equivalent, then their Hermitian duals are equivalent as well, and hence their Hermitian hulls have the same dimension. Therefore, the dimension of the Hermitian hull is a well-defined invariant of the equivalence classes considered in Table~\ref{tab:hermitian-hull}. In the table, $\ell_i$ denotes the number of codes $\mathcal{C}$ with $\dim(\mathcal{H}_H(C))=i$.
Throughout this section we count only codes with no zero coordinate.
This restriction is not preserved by Hermitian duality, so that the entries for $[n,k]_4$ and $[n,n-k]_4$ codes need not coincide, in contrast with $A^H_{n,k,\ell,q^2}$.

\begin{longtable}{c c c r r r r r r r}
\caption{Distribution of linear codes over $\mathrm{GF}(4)$ according to the dimension of their Hermitian hull.}\label{tab:hermitian-hull}\\
\hline
$n$ & $k$ & $d$ & Total & $\ell_0$ & $\ell_1$ & $\ell_2$ & $\ell_3$ & $\ell_4$ & $\ell_5$ \\
\hline
\endfirsthead
\hline
$n$ & $k$ & $d$ & Total & $\ell_0$ & $\ell_1$ & $\ell_2$ & $\ell_3$ & $\ell_4$ & $\ell_5$ \\
\hline
\endhead
6 & 2 & 1 & 9 & 5 & 2 & 2 & 0 & 0 & 0 \\
6 & 3 & 1 & 24 & 7 & 12 & 3 & 2 & 0 & 0 \\
6 & 4 & 1 & 17 & 8 & 5 & 4 & 0 & 0 & 0 \\
6 & 5 & 1 & 5 & 2 & 3 & 0 & 0 & 0 & 0 \\
7 & 2 & 1 & 12 & 5 & 6 & 1 & 0 & 0 & 0 \\
7 & 3 & 1 & 55 & 23 & 20 & 10 & 2 & 0 & 0 \\
7 & 4 & 1 & 70 & 25 & 30 & 11 & 4 & 0 & 0 \\
7 & 5 & 1 & 28 & 13 & 10 & 5 & 0 & 0 & 0 \\
7 & 6 & 1 & 6 & 3 & 3 & 0 & 0 & 0 & 0 \\
8 & 2 & 1 & 17 & 9 & 4 & 4 & 0 & 0 & 0 \\
8 & 3 & 1 & 126 & 43 & 63 & 14 & 6 & 0 & 0 \\
8 & 4 & 1 & 301 & 125 & 104 & 61 & 8 & 3 & 0 \\
8 & 5 & 1 & 184 & 63 & 87 & 24 & 10 & 0 & 0 \\
8 & 6 & 1 & 44 & 21 & 14 & 9 & 0 & 0 & 0 \\
8 & 7 & 1 & 7 & 3 & 4 & 0 & 0 & 0 & 0 \\
9 & 2 & 1 & 22 & 9 & 11 & 2 & 0 & 0 & 0 \\
9 & 3 & 1 & 286 & 127 & 112 & 41 & 6 & 0 & 0 \\
9 & 4 & 1 & 1419 & 611 & 581 & 185 & 38 & 4 & 0 \\
9 & 5 & 1 & 1594 & 693 & 622 & 232 & 40 & 7 & 0 \\
9 & 6 & 1 & 453 & 181 & 195 & 61 & 16 & 0 & 0 \\
9 & 7 & 1 & 65 & 30 & 24 & 11 & 0 & 0 & 0 \\
9 & 8 & 1 & 8 & 4 & 4 & 0 & 0 & 0 & 0 \\
10 & 2 & 1 & 29 & 16 & 7 & 6 & 0 & 0 & 0 \\
10 & 3 & 1 & 640 & 255 & 307 & 59 & 19 & 0 & 0 \\
10 & 4 & 1 & 7970 & 4015 & 2881 & 966 & 90 & 18 & 0 \\
10 & 5 & 1 & 22405 & 10598 & 9495 & 1916 & 368 & 23 & 5 \\
10 & 6 & 1 & 9278 & 4581 & 3391 & 1157 & 124 & 25 & 0 \\
10 & 7 & 1 & 1071 & 427 & 491 & 118 & 35 & 0 & 0 \\
10 & 8 & 1 & 93 & 45 & 31 & 17 & 0 & 0 & 0 \\
10 & 9 & 1 & 9 & 4 & 5 & 0 & 0 & 0 & 0 \\
11 & 2 & 1 & 36 & 16 & 17 & 3 & 0 & 0 & 0 \\
11 & 3 & 1 & 1431 & 683 & 575 & 154 & 19 & 0 & 0 \\
11 & 4 & 1 & 51456 & 27184 & 19753 & 4112 & 379 & 28 & 0 \\
\hline
\end{longtable}

The distribution in Table~\ref{tab:hermitian-hull} exhibits a clear concentration towards small Hermitian hull dimensions. In most rows, the largest numbers occur for hull dimensions $0$ and $1$, with either $\ell_0$ or $\ell_1$ being the dominant class. Starting from $\ell_2$, the numbers generally decrease rapidly as the hull dimension increases. In particular, the classes corresponding to the largest possible hull dimensions are relatively small; for example, for $[11,4]_4$ codes the numbers are
$(27184,19753, 4112, 379, 28)$.
Thus, although codes with large Hermitian hulls do occur, they represent only a small proportion of the classified codes. The same general tendency is visible throughout the table, with the distribution becoming strongly concentrated in the classes with hull dimension $0$ and $1$.

For $k=1$ and $k=n-1$ the numbers of inequivalent codes can be determined explicitly. In the statement below, $N_{n,k,i}$ denotes the number of inequivalent $[n,k]_4$ codes without zero coordinates whose Hermitian hull has dimension $i$.

\begin{proposition}\label{prop:boundary}
Let $n\ge 2$. Then
\begin{itemize}
\item[\textup{(i)}] $(N_{n,1,0},N_{n,1,1})=(0,1)$ if $n$ is even, and $(1,0)$ if $n$ is odd;
\item[\textup{(ii)}] $N_{n,n-1,0}=\left\lfloor\frac{n-1}{2}\right\rfloor$ and $N_{n,n-1,1}=\left\lceil\frac{n-1}{2}\right\rceil$.
\end{itemize}
In particular $N_{n,k,0}<N_{n,k,1}$ for even $n$ and $k\in\{1,n-1\}$, that is, on the family of parameters excluded in Corollary~\ref{cor:Hstrong}.
\end{proposition}

\begin{proof}
Let $v\in\F_{4}^{n}$ be a nonzero vector and $C=\langle v\rangle$. Every $[n,n-1]_4$ code equals $C^{\perp_H}$ for a unique one-dimensional code $C$, and $C^{\perp_H}$ determines $C$. Two one-dimensional codes are equivalent if and only if their generators have the same weight. Since equivalent codes have equivalent Hermitian duals, the equivalence classes in dimensions $1$ and $n-1$ are both indexed by
$w=\operatorname{wt}(v)$.  Moreover, $\mathcal{H}_H(C)=C\cap C^{\perp_H}=
\mathcal{H}_H(C^{\perp_H})$, whence the two classes indexed by $w$ have the same hull dimension. Since $x\bar{x}=x\cdot x^{2}=x^{3}=1$ for every $x\in\F_4^{*}$, we have $\langle v,v\rangle_H=\mathrm{wt}(v)\cdot 1$ in $\F_4$, and therefore $\dim\mathcal{H}_H(C)=1$ if $w$ is even and $0$ if $w$ is odd.
The code $C$ has no zero coordinate exactly when $w=n$, which gives (i). The code $C^{\perp_H}$ has no zero coordinate exactly when $w\ge 2$; the admissible weights are therefore $w=2,\dots,n$, and counting the odd and the even ones among them gives $N_{n,n-1,0}=\lfloor (n-1)/2\rfloor$ and $N_{n,n-1,1}=\lceil (n-1)/2\rceil$. For even $n$ the number $n-1$ is odd and the two values differ, which proves the last assertion.
\end{proof}

\section{Conclusion}

In this paper, we investigate the enumeration of linear codes with prescribed Hermitian and symplectic hull dimensions, extending known results for Euclidean hulls. We also compare the numbers of linear codes with the same length and dimension but different types of hulls, such as Euclidean, Hermitian, and symplectic hulls, and we complement the theoretical results with a computational classification of quaternary linear codes with respect to the Hermitian inner product. We focus on the quaternary case since $\mathbb{F}_4$ is the only finite field for which the equivalence relation preserves the Hermitian inner product which means that equivalent codes have equivalent Hermitian hulls.

Several directions for future research emerge from our results. One such direction is to study in more detail how the choice of inner product affects the distribution of hull dimensions for fixed code parameters, which may reveal deeper structural differences between these classes of codes. Another promising direction is to extend the computational classification to quaternary codes of larger lengths and dimensions, with special emphasis on Hermitian LCD and Hermitian self-orthogonal codes. For this purpose, we will use the program \textsc{Generation} from the software package \textsc{QExtNewEdition} \cite{Generation}.

\section*{Acknowledgements}

S. Bouyuklieva and I. Bouyukliev are supported by Bulgarian Science Fund under Contract KP-06-H92/3 - 08.12.2025.
K. Ishizuka is supported by JSPS KAKENHI Grant Number JP25K17290.
F. \"{O}zbudak is supported by T\"{U}B\.{I}TAK project 223N065.

\backmatter



\end{document}